\documentclass[12 pt]{article}
\usepackage{latexsym, amssymb, amsmath, amsthm, geometry, changebar}

\numberwithin{equation}{section}

\begin{document}

\newcommand{\Z}{\mathcal{Z}}
\newcommand{\B}{\mathcal{B}}
\newcommand{\A}{\mathcal{A}}
\newcommand{\E}{\mathcal{E}}
\newcommand{\G}{\mathcal{G}}
\newcommand{\Q}{\mathcal{Q}}
\newcommand{\zlj}{Z_{\ell, j}}
\newcommand{\Lp}{\mathcal{L}}
\newcommand{\ppi}{\mathcal{P}_\pi}
\newcommand{\HDelta}{\hat{\Delta}}
\newcommand{\hDelta}{\hat{\Delta}}
\newcommand{\F}{\hat{F}}
\newcommand{\T}{\hat{T}}
\newcommand{\I}{\hat{I}}
\newcommand{\X}{\hat{X}}
\newcommand{\V}{\mathcal{V}}
\newcommand{\M}{\mathcal{M}}
\newcommand{\Ho}{\mathcal{H}}
\newcommand{\N}{\mathbb{N}}
\newcommand{\Na}{\mathcal{N}}
\newcommand{\R}{\mathbb{R}}
\newcommand{\ff}{\mathcal{F}}
\newcommand{\ab}{\bar{a}}
\newcommand{\ve}{\varepsilon}
\newcommand{\vf}{\varphi}
\newcommand{\tm}{\tilde{m}}
\newcommand{\tg}{\tilde{g}}
\newcommand{\tf}{\tilde{f}}
\newcommand{\tp}{\tilde{\varphi}}
\newcommand{\tG}{\tilde{G}}
\newcommand{\tH}{\tilde{H}}
\newcommand{\td}{\tilde{d}}
\newcommand{\tx}{\tilde{x}}
\newcommand{\ty}{\tilde{y}}
\newcommand{\tnu}{\tilde{\nu}}
\newcommand{\tmu}{\tilde{\mu}}
\newcommand{\teta}{\tilde{\eta}}
\newcommand{\trho}{\tilde{\rho}}
\newcommand{\accim}{{\em a.c.c.i.m.}}
\newcommand{\re}{\mbox{Re}}
\newcommand{\Arg}{\mbox{Arg}}
\newcommand{\Crit}{\mbox{Crit}}
\newcommand{\sCrit}{\mbox{\scriptsize Crit}}
\newcommand{\mi}{\mbox{\tiny min}}
\newcommand{\ma}{\mbox{\tiny max}}
\newcommand{\hlj}{H_{\ell,j}}

\newtheorem{theorem}{Theorem}[section]
\newtheorem{lemma}[theorem]{Lemma}
\newtheorem{sublemma}[theorem]{Sublemma}
\newtheorem{proposition}[theorem]{Proposition}
\newtheorem{remark}[theorem]{Remark}
\newtheorem{definition}[theorem]{Definition}
\newtheorem{cor}[theorem]{Corollary}
\newtheorem*{thma}{Theorem A}
\newtheorem*{thmb}{Theorem B}

\title{Existence and convergence properties of physical measures for certain
dynamical systems with holes.}
\author{Henk Bruin\thanks{HB was supported in part by
EPSRC grants GR/S91147/01 and EP/F037112/1} 
\and Mark Demers\thanks{MD was supported in
part by EPSRC grant GR/S11862/01 and NSF grant DMS-0801139.} 
\and Ian Melbourne\thanks{IM was supported in part by
EPSRC grant GR/S11862/01 
\newline MD thanks the University of Surrey for an engaging
visit during which this project was started.  In addition, HB would like to thank
Georgia Tech; MD would like to thank the Scuola Normale Superiore, Pisa; MD and IM would
like to thank MSRI, Berkeley, where part of this work was done.}}

\maketitle

\begin{abstract}
We study two classes of dynamical systems with
holes:  expanding maps of the interval and Collet-Eckmann maps
with singularities.
In both cases, we prove that there is a natural absolutely continuous
conditionally invariant measure $\mu$ (\accim) with
the physical property that strictly positive H\"{o}lder continuous
functions converge to the density of $\mu$ under the renormalized dynamics
of the system.  In addition, we construct an invariant measure
$\nu$, supported on the Cantor set of points that never escape
from the system, that is ergodic and enjoys exponential decay of correlations
for H\"{o}lder observables.  We show that $\nu$ satisfies
an equilibrium principle which implies
that the escape rate formula, familiar to the thermodynamic formalism,
holds outside the usual setting.  In particular, it holds for Collet-Eckmann
maps with holes, which are not uniformly hyperbolic and do not admit a
finite Markov partition.

We use a general framework of Young towers with holes
and first prove results about the \accim\ and the invariant measure on the
tower.
Then we show how to transfer results to the original dynamical system.
This approach can be expected to generalize to other dynamical systems than
the two above classes.
\iffalse
\\[0.3cm]
{\bf R\'esum\'e:}
Nous \'etudions deux classes de syst\`emes dynamiques avec trous:
applications expansives sur l'intervalle et applications de Collet-Eckmann.
Dans les deux cas nous d\'emontrons l'existence d'une mesure naturelle
absolument continue conditionellement invariante $\mu$ (\accim)
poss\'edant la propri\'et\'e physique que toute fonction h\"olderienne
continue et strictement positive converge vers la densit\'e d\'efinie
par $\mu$ sous la dynamique renormalis\'ee du syst\`eme.\\
Additionellement, nous construisons une mesure invariante, $\nu$,
\`a support l'ensemble de Cantor constitu\'e des points ne s'\'echappant
jamais du syst\`eme, ergodique et \`a d\'ecorr\'elation exponentielle
pour toute observable h\"olderienne. Nous d\'emontrons un principe
d'\'equilibre pour $\nu$ impliquant la persistance d'une formule de taux
d'\'evasion, famili\`ere dans le formalisme thermodynamique, en
dehors des conditions usuelles. En particulier, cette formule est
v\'erifi\'ee pour les applications de Misiurewicz avec trous, qui ne
sont pas uniform\'ement hyperboliques et n'admettent pas de partition
de Markov finie.\\
Nous utilisons le contexte g\'en\'erale des tours de Young avec trous,
et, dans un premier temps, nous d\'emontrons les r\'esulats de
%l'existence de $\mu$
l'\accim\ %et construisons $\nu$,
et la mesure invariante sur les tours.
Nous montrons ensuite comment ramener ces r\'esultats au syst\`eme dynamique
original. Cette approche peut \^etre attendue \`a se g\'en\'eraliser
\`a d'autres classes de syst\`emes que ceux consid\'er\'es ici.
\fi
\end{abstract}

\section{Introduction}
\label{introduction}

Dynamical systems with holes are examples of systems whose domains are not
invariant under the dynamics.  Important questions in the study
of such open systems include: what is the escape rate from the phase space
with respect to a given reference measure?  Starting with an initial
probability measure $\mu_0$ and letting $\mu_n$ denote the distribution
at time $n$ conditioned on not having escaped, does $\mu_n$ converge to
some limiting distribution independent of $\mu_0$?  Such a measure,
if it exists, is a \emph{conditionally invariant measure}.

These questions have been addressed primarily for uniformly expanding
or hyperbolic systems which admit finite Markov partitions:
expanding maps on $\R^n$ \cite{pianigiani yorke, collet 1, collet 2};
Smale horseshoes \cite{cencova 1, cencova 2};
Anosov diffeomorphisms
\cite{chernov m 1, chernov m 2, chernov mt 1, chernov mt 2};
billiards with convex scatterers satisfying a non-eclipsing condition
\cite{lopes mark, richardson}; and large parameter logistic maps whose critical
point maps out of the interval \cite{homburg young}.

Requirements on Markov partitions have been dropped
for expanding maps of the interval
\cite{baladi keller, chernov bedem, liverani maume, demers exp};
and more recently for piecewise uniformly hyperbolic maps in two
dimensions \cite{demers liverani}.
Nonuniformly
hyperbolic systems have been studied in the form of logistic
maps with generic holes \cite{demers logistic}.
Typically a restriction on the size of the hole is introduced
in order to control the dynamics.

A central object of study in these open
systems is the conditionally invariant
measure mentioned previously.  Given a self-map $\T$ of a measure
space $\X$, we identify a set $H \subset \X$ which we call the \emph{hole}.
Once the
image of a point has entered $H$, we do not allow it to return.
Define $X = \X \backslash H$ and
$T = \T |_{X \cap \T^{-1}X}$. A probability measure
$\mu$ is called \emph{conditionally invariant} if it satisfies
\[
\mu(A) = \frac{\mu(T^{-1}A)}{\mu(T^{-1}X)}
\]
for each Borel $A \subseteq X$.  Iterating this relation and setting
$\lambda = \mu(T^{-1}X)$, we see that $\mu(T^{-n}A) = \lambda^n \mu(A)$.
The number $\lambda$ is called the eigenvalue of $\mu$ and $-\log \lambda$
represents its \emph{exponential rate of escape} from $X$.

If $\mu$ is absolutely continuous with respect to a reference measure $m$,
we call $\mu$ an absolutely continuous conditionally invariant measure
and abbreviate it by \accim

In \cite{demers exp} and \cite{demers logistic}, the author constructed
Young towers to study expanding maps of the interval and
unimodal Misiurewicz maps
with small holes.  The systems were shown to admit an \accim\ with a density unique
in a certain class of densities and converging to the SRB measure of the
closed system as the diameter of the hole tends to zero.  However, left
open in these papers was the question of what class of measures converges to
the \accim\ under the (renormalized) dynamics of $T$.
This question is especially important for open systems since even for
well-behaved hyperbolic systems,
many \accim\ may exist with overlapping supports and arbitrary escape rates
\cite{demers young}.  Thus it is essential to distinguish a natural \accim\
which attracts a reasonable class of measures, including the reference measure.

The purpose of this paper is two-fold.  First, we prove that for a large class of
systems with holes, including
\begin{enumerate}
\item $C^{1+\alpha}$ expanding maps of the interval (see Theorem~\ref{thm:exp convergence}),
and
\item multimodal Collet-Eckmann
maps with singularities (see Theorem~\ref{thm:CE convergence}),
\end{enumerate}
all H\"older continuous densities $f$ which are bounded away from
zero converge exponentially to the \accim\
under the renormalized dynamics of $T$.  To be precise,
if $\Lp$ is the transfer operator associated with $T$ and $|\cdot |_1$
the $L^1(m)$-norm, then $\Lp^nf/|\Lp^nf|_1$ converges exponentially
to the density of $\mu$ as $n \to \infty$.  Although
similar results are known for $C^2$ expanding maps with holes
\cite{chernov bedem, liverani maume}, they are completely new for
multimodal maps, and even for unimodal maps without singularities. In addition,
we strengthen the results on the dynamics of the tower which were used
in \cite{demers exp} and \cite{demers logistic}.

Second, we study the set of nonwandering points of each system:
the (measure zero) set of points, $X^\infty$, which never enter the hole.
We construct
an ergodic invariant probability measure $\tnu$ supported on $X^\infty$ which
enjoys exponential decay of correlations on H\"{o}lder functions.
The measure $\tnu$ is characterized by a physical limit
and satisfies an equilibrium principle.  This implies the generalized
\emph{escape rate formula} for both classes of systems in question,
\[
\log \lambda = h_{\tnu}(T) - \int_X \log JT \, d\tnu
\]
where $\lambda$ represents the exponential rate of escape from $X$ with respect to
the reference measure $\tm$, $h_{\tnu}(T)$ is the metric
entropy of $T$ with respect to $\tnu$, and $JT$ is the Jacobian of $T$
with respect to $\tm$.

This formula is
well-known when the usual thermodynamic formalism applies (in the
presence of a finite Markov partition)
\cite{bowen, cencova 1, chernov m 1, chernov mt 2, collet 1}.  In
\cite{baladi keller}, an equilibrium principle was established for
piecewise expanding maps with generalized potentials of bounded variation.
The paper \cite{bruin keller}
deals with equilibrium states of the unbounded potential
$-t \log|T'|$, $t \approx 1$, for Collet-Eckmann unimodal maps $T$,
using a weighted transfer operator, but not allowing any holes.
Both \cite{baladi keller} and \cite{bruin keller}
use canonical Markov extensions (frequently called
Hofbauer towers).
In Theorem~\ref{thm:T-equilibrium} we generalize those results
to systems with holes having no Markov structure and nonuniform hyperbolicity
by constructing Young towers.
In contrast to previous results, we do not use bounded variation techniques
and so are able to allow potentials which are piecewise H\"{o}lder
continuous.  This answers in the
affirmative a conjecture of Chernov and van dem Bedem regarding expanding
maps with holes \cite{chernov bedem}
and a more general question raised in \cite{demers young}.

\begin{remark}
It is important to note that the Young towers must be constructed for each system
{\em after} the introduction of holes since the presence of holes affects return
times in a possibly unbounded way.
Thus existing tower constructions for the corresponding closed systems cannot
be used directly.  
\end{remark}

Throughout the paper, we emphasize the physical properties of the measures
involved and their characterization as push forward and pull back limits
under the renormalized dynamics.  In particular, the
measures are independent of the Markov extensions used.

In Section~\ref{results}, we formulate our results precisely and include
a brief discussion of the issues involved.  Section~\ref{quasi-compact}
proves the convergence results on the tower while Section~\ref{applications}
applies these results to two classes of concrete systems with holes: expanding
maps of the interval and Collet-Eckmann maps with singularities.
Section~\ref{equilibrium}
contains proofs of the equilibrium principles for both the tower and
the underlying dynamical system.

%%%%%%%%%%%%%%%%%%%%%%%%%%%%%%%%%%%%%%%%%%
%%%%%%%%%%%%%%%%%%%%%%%%%%%%%%%%%%%%%%%%%%

\section{Setting and Statement of Results}
\label{results}

\subsection{Young Towers}

We recall the definition of a Young tower.  Let $\HDelta_0$ be a measure
space and let $\Z_0$ be a countable measurable
partition of $\HDelta_0$.  Given a finite reference measure $m$ on $\HDelta$, let $R$ be a function on
$\HDelta_0$ which is
constant
on elements of the partition and for which $\int R \, dm < \infty$.  We define the tower over $\HDelta_0$ as
\[
\HDelta = \{ (x, n) \in \Delta_0 \times \N : n < R(x) \},
\]
where $\N = \{ 0,1,2,\dots\}$.
We call $\HDelta_\ell = \HDelta|_{n=\ell}$ the $\ell^{th}$ level of the tower.  The action of the tower map
$\F$ is characterized by
\[
\begin{array}{rll}
\F(x,n) = & (x, n+1) & \mbox{if $n+1 < R(x)$} \\
\F^{R(x)}(\Z(x)) = & \bigcup_{j \in J_x} Z_{j} &
\mbox{for some subset of partition elements of $\Z_0$ indexed by $J_x$}
\end{array}
\]
where $\Z_0(x)$ is the element of $\Z_0$ containing $x$ and
$\F^{R(x)}|_{\Z_0(x)}$ is injective.

We will abuse notation slightly and refer to a point $(x,n)$ in the tower as simply $x$ and $\HDelta_n$ will be
made clear by the context.  Also, the partition $\Z_0$
and the action of $F$ induce a natural partition of $\HDelta$
which we shall refer to by $\Z$, with elements $Z_{\ell, j}$ in $\HDelta_\ell$.
With this convention, it is clear that $\Z$ is a Markov partition for $\F$.
The definition of $R$ extends easily to the entire tower as well: $R(x)$ is simply the first time that
$x$ is mapped to $\HDelta_0$ under $\F$.  We extend $m$ to each level of the tower by setting
$m(A) = m(\F^{-\ell}A)$ for every measurable set $A \subset \HDelta_\ell$.

%%%%%%%%%%%%%%%%%%%%%%%%%%%%%%%%%%%%%%%%

\subsubsection{Introduction of Holes}
\label{holes}

We define a hole $H$ in $\HDelta$ as the union of countably many elements of the partition $\Z$, i.e.,
$H = \bigcup H_{\ell, j}$ where each $H_{\ell, j} = Z_{\ell, k}$
for some $k$.
Also set $H_\ell = \sum_j H_{\ell,j}=H\cap\HDelta_\ell$.
This preserves the Markov
structure of the returns to $\HDelta_0$, but the definition of the return time function $R$ needs a slight
modification:
if $x$ is mapped into $H$ before it reaches $\HDelta_0$, $R(x)$ is defined to be the time that $x$ is mapped
into $H$; otherwise, $R(x)$ remains unchanged.  If $Z_{\ell, j} \subset H$, then all the elements of $\Z$ directly
above $Z_{\ell, j}$ are deleted since once $\F$ maps a point into $H$, it disappears forever.

We will be interested in studying the dynamics of the points which have not yet fallen into the hole.  To this end,
we define $\Delta = \HDelta \backslash H$ and $\Delta^n = \bigcap_{i = 0}^n \F^{-i} \Delta$, so  $\Delta^n$ is the
set of points which have not fallen into the hole by time $n$.  Define the map $F = \F |_{\Delta^1}$
and its iterates by $F^n = \F^n |_{\Delta^n}$.  We denote by $Z_{\ell, j}^* \subset \Delta$ those elements of $\Z$ for which
$\F(Z) \subset \HDelta_0$. In this
paper we will study the map $F$ and the transfer operator associated with it.

We consider towers with the following properties.
\begin{enumerate}
  \item[{\bf(P1)}]  {\em Exponential returns.}  There exist constants $C>0$ and $\theta < 1$ such that
    $m(\HDelta_n) \leq C\theta^n$.
  \item[{\bf(P2)}]  {\em Generating partition.}  For each $x \neq y \in \Delta$, there
    exists a separation time $s(x,y)<\infty$
such that $s(x,y)$ is the smallest nonnegative integer $k$ such that
    $F^k(x)$ and $F^k(y)$ lie in different elements of $\Z$ or $\F^k(x)$, $\F^k(y) \in H$.
  \item[{\bf(P3)}]  {\em Finite images.}  Let $\Z_0^{im}$ be the partition of $\Delta_0$ generated by the sets
$\{F^RZ\}_{Z \in \Z_0}$.  We require that $\Z_0^{im}$ be a finite partition.
\end{enumerate}

Due to (P3) we define $c_0 := \min_{Z' \in \Z_0^{im}} m(Z')> 0$.

Using property (P2), we define a metric on $\Delta$ by
$d(x,y) = \beta^{s(x,y)}$ for some $\beta\in(\theta,1)$ where
$\theta$ is as in (P1).   (The value of $\beta$ may
be further restricted depending on the underlying dynamical system to
which we wish to apply the tower.)

We say that $(F, \Delta)$ is \emph{transitive} if for each
$Z'_1, Z'_2 \in \Z_0^{im}$, there exists an $n \in \N$ such that
$F^n(Z'_1) \cap Z'_2 \neq 0$.  We say that $F$ is \emph{mixing}
if for each $Z' \in \Z_0^{im}$, there is an $N$ such that
$\Delta_0 \subset F^n(Z')$ for all $n \geq N$.

\begin{remark}
We define mixing in this way because the usual requirement,
$\gcd (R|_{\Delta_0})=1$, made for towers
with a single base (i.e., $\Z_0^{im}$ contains a single element) is not sufficient
to eliminate periodicity in towers with multiple bases.
\end{remark}

Since we may always construct a tower with no holes in the base (by simply
choosing a reference set in the underlying system which does not intersect
the hole), we consider
towers with no holes in $\Delta_0$.  Define
\[
q := \sum_{\ell \geq 1} m(H_\ell) \beta^{-(\ell -1)}.
\]
Our assumption on the size of the hole is,

\medskip
\noindent
\parbox{.1 \textwidth}{\bf(H1)}
\parbox[t]{.8 \textwidth}{ $\displaystyle q < \frac{(1-\beta)c_0}{1+C_1}$ }
\medskip

\noindent
where $C_1$ is the distortion constant of equation~\eqref{eq:distortion}
below.

\begin{remark}
If one is interested in considering towers with holes in the base, then the
definition of $q$ is modified to be
$q :=  \sum_{\ell \geq 1} m(H_\ell) \beta^{-(\ell -1)} + c_0^{-1}(1+C_1) m(H_0)
\sum_{Z_{\ell, j}^*} m(Z_{\ell, j}^*) \beta^{-\ell}$.
Assumption (H1) remains the same and all the results of this paper apply.
\end{remark}

%%%%%%%%%%%%%%%%%%%%%%%%%%%%%%%%%%%%%%

\subsubsection{Transfer Operator}

In order to study the evolution of densities according to the dynamics of
$(F, \Delta)$, we introduce the
transfer operator $\Lp_F$ defined on $L^1(\Delta)$ by
\[
\Lp_Ff(x) = \sum_{y \in F^{-1}x} f(y)g(y)
\]
where $g = \frac{dm}{d (m \circ F)}$.
Unless otherwise noted, we will refer to $\Lp_F$ as simply $\Lp$ for the
rest of this paper.  Higher iterates of $\Lp$ are given by
\[
\Lp^n f(x) = \sum_{y \in F^{-n}x} f(y)g_n(y)
\]
where  $g_n = g \cdot g \circ F \cdots g \circ F^{n-1}$.
For $f \in L^1(\Delta)$, we define the Lipschitz constant of $f$ to be
\[
\mbox{Lip}(f) = \sup_{\ell, j} \mbox{Lip}(f_{\ell,j})
\quad \mbox{ and } \quad
\mbox{Lip}(f_{\ell,j}) =
\sup_{x \neq y \in Z_{\ell,j}} \frac{|f(x)-f(y)|}{d(x,y)}.
\]
We will assume that $\mbox{Lip}(\log g) <\infty$.
This assumption on $g$ implies the following standard distortion estimate.

There exists a constant $C_1>0$ such that for every $n >0$ and for all $x,y \in \Delta$ such that $s(x,y) \geq n$,
we have
\begin{equation}
\label{eq:distortion}
\left| \frac{g_n(x)}{g_n(y)} - 1\right| \leq C_1d(F^nx, F^ny).
\end{equation}
In particular, if $E_n(y)$ denotes the $n$-cylinder containing $y$, then
\begin{equation}
\label{eq:distortion2}
g_n(y) \leq (C_1+1) \frac{m(E_n(y))}{m(F^nE_n(y))}\leq (C_1+1)c_0^{-1}m(E_n(y)).
\end{equation}

It is easy to see that $d\mu = \vf\, dm$ is an {\em a.c.c.i.m.}\ with
eigenvalue $\lambda$ if and only if $\Lp \vf = \lambda \vf$.
Simply write for any measurable set $A \subset \Delta$,
\[
\mu(F^{-1}A) = \int_{F^{-1}A} \vf \, dm = \int_A \Lp \vf \, dm,
\qquad \mbox{and} \qquad
\lambda \mu(A) = \lambda \int_A \vf \, dm.
\]
Then the two left hand sides are equal if and only if the two right
hand sides are equal.
Thus the
properties of {\em a.c.c.i.m.}\ for $(F, \Delta)$ are tied to the
spectral properties of $\Lp$.

%%%%%%%%%%%%%%%%%%%%%%%%%%%%%%%%%%%

\subsection{First Results:  a Spectral Gap for $\Lp$}

We begin by proving a spectral decomposition for $\Lp$ corresponding to
$(F,\Delta)$ acting on a certain Banach space of functions.
The result follows essentially from Proposition~\ref{prop:lasota-yorke}
using estimates similar to those in
\cite{young} and \cite{demers exp}.  One important difference in the present
setting is that $\Lp$ does not have spectral radius 1, as it does for systems
without holes, so careful estimates are needed to ensure that
a discrete spectrum exists outside the disk of radius $\beta<1$.

%%%%%%%%%%%%%%%%%%%%%%%%%%%%%%%%%%%%%%%%%%

\subsubsection{Definition of the Banach space}

Let $\V(\Delta)$ be the set of functions on $\Delta$ which are Lipschitz continuous on elements
of the partition $\Z$.  For each $Z_{\ell, j}$ and $f \in \V(\Delta)$,
we set $f_{\ell, j} = f|_{\zlj}$.
We denote by $|f|_\infty$ the $L^\infty$ norm of $f$ and define
\[
\|f_{\ell, j}\|_\infty  :=  |f_{\ell, j}|_\infty \beta^\ell,
\qquad \qquad
\|f_{\ell, j}\|_{\mbox{\tiny Lip}}   :=  \mbox{Lip}(f_{\ell, j}) \beta^\ell
\]
and
\[
||f|| = \max\{\|f\|_\infty, \|f\|_{\mbox{\tiny Lip}} \}
\]
where $\|f\|_\infty = \sup_{\ell, j} \|f_{\ell, j}\|_\infty$ and
$\|f\|_{\mbox{\tiny Lip}} = \sup_{\ell, j} \|f_{\ell, j}\|_{\mbox{\tiny Lip}}$.

Our Banach space is then $\B = \{ f \in \V(\Delta) : \|f\| < \infty \}$.
The choice $\beta\in(\theta,1)$
(where $\theta$ comes from condition (P1)) guarantees that
$\B \subset L^1(\Delta)$ and the
unit ball of $\B$ is compactly embedded in $L^1(\Delta)$.
The proof of this fact is
similar to that in \cite[Proposition 2.2]{demers exp}.

%%%%%%%%%%%%%%%%%%%%%%%%%%%%%%%%%%%%%%%

\subsubsection{Spectral picture and convergence results}

Let $|\cdot|_1$ denote the $L^1$-norm with respect to $m$.
In Section~\ref{lasota-yorke}, we prove the following.

\begin{proposition}
\label{prop:lasota-yorke}
Let $(F, \Delta, H)$ be a tower with holes satisfying properties (P1)-(P3) and assumption (H1).
Then there exists $C>0$ such that for each $n \in \mathbb{N}$ and all
$f \in \B$,
\begin{eqnarray*}
\| \Lp^nf \| & \leq & C \beta^n \|f\|_{\mbox{\tiny\rm Lip}} + C |f|_1.
\end{eqnarray*}
\end{proposition}

Proposition~\ref{prop:lasota-yorke}, together with the compactness of the
unit ball of $\B$ in $L^1(m)$ and the fact that $|\Lp^nf|_1 \leq |f|_1$,
are enough
to conclude that $\Lp:\B \circlearrowright$
has essential spectral radius bounded by
$\beta$ and spectral radius bounded by 1 \cite{baladi}.
However, since the system is open, we expect the actual spectral
radius of $\Lp$ to be a constant $\lambda < 1$.  We must show that
$\lambda>\beta$
in order to conclude
that there is a spectral gap.  This fact is proved in
Section~\ref{spectral gap} using assumption
(H1) on the measure of the hole.

Once a spectral gap has been established, the next proposition
shows that the familiar spectral picture holds true for the open
system.  This is proved in Section~\ref{spectral gap}.

\begin{proposition}
\label{prop:spectral picture}
The spectral radius of $\Lp$ on $\B$ is $\lambda > \beta$ and $\Lp$
is quasi-compact as an operator on $\B$.  In addition,
\begin{enumerate}
  \item[(i)]  If $F$ is mixing, then $\lambda$ is a simple eigenvalue and
    all other eigenvalues have modulus strictly less than $\lambda$.
    Moreover, there exists $\delta>0$ such that the unique probability
    density $\vf$ corresponding to $\lambda$ satisfies
    $\delta \lambda^{-\ell} \leq \vf \leq \delta^{-1} \lambda^{-\ell}$,
    on each $\Delta_\ell$.
  \item[(ii)] If $F$ is transitive and periodic with period $p$, then
    the set of eigenvalues of modulus $\lambda$ consists
    of simple eigenvalues $\{ \lambda e^{2\pi i k/p} \}_{k=0}^{p-1}$.
    The unique probability density corresponding to $\lambda$ satisfies
    the same bounds as in (i).
  \item[(iii)] In general, $F$ has finitely many transitive
    components, each with its own largest eigenvalue $\lambda_j$.  On
    each component, (ii) applies.
\end{enumerate}
\end{proposition}
Since Proposition \ref{prop:spectral picture} eliminates the
possibility of generalized eigenvectors, the projection
$\Pi_\lambda$ onto the eigenspace of eigenvalue $\lambda$ is
characterized for each $f \in \B$ by the limit
\[
\Pi_\lambda f
= \lim_{n \to \infty} \frac{1}{n}\sum_{k=0}^{n-1} \lambda^{-k} \Lp^kf
\]
where convergence is in the $\| \cdot \|$-norm.
By (iii), the eigenspace $\mathbb{V}_\lambda := \Pi_\lambda \B$ has a finite
basis of probability densities, each representing an \accim\ with
escape rate $-\log \lambda$.

\begin{cor}
\label{cor-exp}
Suppose that $F$ is mixing and let $\vf \in \mathbb{V}_\lambda$
denote the unique probability density given by (i).
Then there exists $\sigma\in(0,1)$ and $C\ge1$ such that
\[
\|\lambda^{-n}\Lp^nf-c(f)\vf\|\leq C\|f\|\sigma^n,
\]
for all $f\in\B$ where $c(f)$ is a constant depending on $f$.
\end{cor}

\begin{proof}
The operator $\lambda^{-1}\Lp:\B\to\B$ has spectral radius $1$ and
essential spectral radius $\beta\lambda^{-1}$.
Moreover, there is a simple eigenvalue at $1$ with eigenspace
$\mathbb{V}_\lambda$
spanned by $\vf$ and no further eigenvalues on the unit circle.
Hence there is an $\Lp$-invariant closed
splitting $\B=\mathbb{V}_\lambda\oplus \mathbb{W}_\lambda$ and
$\Lp:\mathbb{W}_\lambda\to \mathbb{W}_\lambda$
has spectral radius $\rho\in(\beta\lambda^{-1},1)$.
The result follows for any $\sigma\in(\rho,1)$.
\end{proof}

The expression in Corollary~\ref{cor-exp} is not satisfactory,
however, if one wishes to obtain
an approximation of the conditionally invariant measure when the eigenvalue
is not known in advance.  In such a case, the object of interest
is the limit $\frac{\Lp^nf}{|\Lp^nf|_1}$ as $n \to \infty$
where the density is renormalized at each step.

\begin{proposition}
\label{prop:convergence}
Let $(F, \Delta)$ be mixing and satisfy the hypotheses of
Proposition~\ref{prop:lasota-yorke}, and let $f \in \B$.
Then $c(f) > 0$ if and only if
\[
\lim_{n \to \infty} \frac{\Lp^nf}{|\Lp^nf|_1} = \vf
\]
where convergence is in the $\| \cdot \|$-norm.
Moreover convergence is at the rate $\sigma^n$ where $\sigma$ is as in
Corollary~\ref{cor-exp}.
\end{proposition}
\begin{proof}
Note that $\lambda^{-n}|\Lp^n f|_1 \to |c(f)|$ by Corollary~\ref{cor-exp}
so that if $c(f) > 0$, we may write
\[
\lim_{n \to \infty} \frac{\Lp^nf}{|\Lp^nf|_1}
   = \lim_{n \to \infty} \frac{\Lp^nf}{\lambda^n} \frac{\lambda^n}{|\Lp^nf|_1}
= \varphi.
\]
The converse follows from the linear structure of $\Lp$.
We write $\B = \mathbb{V}_\lambda \oplus \mathbb{W}_\lambda$ as in the
proof of Corollary~\ref{cor-exp}.
Then $\mathbb{W}_\lambda = \{ g \in \B : c(g) = 0 \}$.
\end{proof}

\begin{remark}
In what follows, we will be interested in establishing which functions
satisfy $c(f)>0$, first on the tower and then for the concrete systems
for which towers are constructed.  Proposition~\ref{prop:convergence criterion}
guarantees that in
particular $c(1) >0$ so that the reference measure on $\Delta$ converges to
the \accim
\end{remark}

%%%%%%%%%%%%%%%%%%%%%%%%%%%%%%%

\subsection{An Equilibrium Principle for $(F, \Delta)$}
\label{tower equilibrium results}

The characterization of $\vf$ in terms of the physical limit
$\Lp^nf/|\Lp^nf|_1$ allows us to construct an invariant
measure $\nu$ singular with respect to $m$ and supported on
$\Delta^\infty = \cap_{n=0}^\infty \Delta^n$,
the set of points which never enter the hole.  Although $\nu$ is
supported on a zero $m$-measure Cantor-like set, the results of this
section indicate that it is physically relevant to the system.

To state our results, we first introduce a new Banach space $\mathcal{B}_0$ consisting
of functions that are uniformly bounded and uniformly locally Lipschitz.   More
precisely, let $|f|_\infty$ denote the standard sup-norm. Then define $|f|_{\mbox{\tiny
Lip}}=\sup_{\ell, j} \mbox{Lip}(f_{\ell,j})$ and $\|f\|_0=\max\{|f|_\infty,
|f|_{\mbox{\tiny Lip}}\}$. Note that contrary to $\|f\|_{\mbox{\tiny Lip}}$, the
seminorm  $|f|_{\mbox{\tiny Lip}}$ doesn't have the weights $\beta^\ell$. Finally, let
\begin{equation}\label{eq:B0}
\B_0 := \{f\in\B:\|f\|_0<\infty\}.
\end{equation}

The following proposition is proved in Section~\ref{invariant measure}.

\begin{proposition}
\label{prop:inv measure}
Suppose $(F, \Delta)$ satisfies properties (P1)-(P3) and (H1) of
Section~\ref{holes} and is mixing.
Then $(F, \Delta)$ admits an invariant probability measure $\nu$
supported on $\Delta^\infty$, which satisfies
\[
\nu(f) = \lim_{n \to \infty} \lambda^{-n} \int_{\Delta^n} f \, d\mu
\]
for all $f \in \B_0$.
In addition, $\nu$ is ergodic and
\[
\Bigl|\int_{\Delta^\infty} f_1 \, f_2 \circ F^n \, d\nu
- \nu(f_1) \nu(f_2) \Bigr| \leq  C \| f_1 \| \, |f_2|_\infty \sigma^n
\]
for all $f_1, f_2 \in\B_0$, $n\ge1$.
\end{proposition}

In addition to its characterization as a limit,
the invariant measure $\nu$ is natural to the system in the sense
that it satisfies the below equilibrium principle.

Let $\nu_0 := \frac{1}{\nu(\Delta_0)} \nu|_{\Delta_0}$ and note that
$\nu_0$ is an invariant measure for $F^R$ on $\Delta^\infty \cap \Delta_0$.
Proposition~\ref{prop:tower equilibrium} shows that in fact $\nu_0$ is a
Gibbs measure for $F^R$.

We call a measure $\eta$ \emph{nonsingular} provided $\eta(F(A)) = 0$
if and only if $\eta(A) = 0$.
The following theorem is proved in
Section~\ref{equilibrium}.

\begin{theorem}
\label{thm:F-equilibrium}
Let $(F, \Delta)$ satisfy the hypotheses of Proposition~\ref{prop:inv measure}.
Let $\M_F$ be the set of $F$-invariant Borel probability
measures on $\Delta$.  Then
\[
\log \lambda = \sup_{\eta \in \M_F} \left\{ h_\eta(F)
  - \int_\Delta \log JF \; d\eta \right\}
\]
where $h_\eta(F)$ is the metric entropy of $\eta$ with respect to $F$ and
$JF$ is the Jacobian of $F$ with respect to $m$.  In addition, $\nu$ is the
unique nonsingular measure in $\M_F$ which attains the supremum.
\end{theorem}

%%%%%%%%%%%%%%%%%%%%%%%%%%

\subsection{Applications to Specific Dynamical Systems}

We apply the results about abstract towers with holes to two specific
classes of dynamical
systems with holes: $C^{1+\alpha}$
piecewise expanding maps of the interval and
locally $C^2$ multimodal Collet-Eckmann maps with singularities.

%%%%%%%%%%%%%%%%%%%%%%%%%%%

\subsubsection{Piecewise Expanding Maps of the Interval}
\label{exp results}

By a piecewise expanding map of the unit interval $\I$, we mean
a map $\T:\I \circlearrowleft$ satisfying the following properties.
There exists a partition of $\I$ into finitely many intervals, $\I_j$,
such that (a) $\T$ is $C^{1+\alpha}$ and monotonic on each $\I_j$
for some $\alpha>0$; and (b) $|\T'| \geq \tau > 2$.
Note that we can always satisfy (b) if $|\T'| \geq 1+\ve$ by
considering a higher iterate of $\T$.

Let $\I^n_j$ denote the intervals of monotonicity for $\T^n$.
The uniform expansion of $\T$ implies the following familiar distortion bound:  there exists a constant
$C_3 >0$ such that for any $n$, if $x$ and $y$ belong to the same $\I^n_j$, then
$\displaystyle
\left| \frac{(\T^n)'(x)}{(\T^n)'(y)} - 1 \right| \leq C_3 |\T^n(x) - \T^n(y)|^\alpha$.

\bigskip
\noindent
{\bf Introduction of Holes.}
A hole $\tH$ in $[0,1]$ is a finite union of open intervals $\tH_j$.
(We use the $\, \tilde{}\,$ to distinguish from the hole on the tower.)
Let $I = \I \backslash \tH$
and for $n \geq 0$, define $I^n = \cap_{i=0}^n \T^{-i}I$.  We are interested in studying the
dynamics of $T^n := \T^n|_{I^n}$.

Let $\gamma$ be the length of the shortest interval of monotonicity of $T$.
Our sole condition on the hole is

\bigskip
\noindent
\parbox{.1 \textwidth}{\bf(H2)}
\parbox[t]{.8 \textwidth}{ $\displaystyle
\tm(\tH) \leq \frac{\gamma^2 (1-\beta) (\tau - 2\beta^{-1})}{1+C_3}$}
\medskip

\noindent
where $\tm$ is Lebesgue measure on $\I$ and
$\beta > \max \{ 2\tau^{-1}, \tau^{-\alpha} \}$.

The following theorem is proved in \cite{demers exp}.

\begin{thma}(\cite{demers exp})
\label{thm:tower existence}
Let $T$ be a $C^{1+\alpha}$ piecewise expanding map of the interval
and let $\tH$ be a hole satisfying the bound given in
(H2).  Then $(T,I)$ admits a tower $(F, \Delta)$ which satisfies properties
(P1)-(P3) and (H1) of Section~\ref{holes}
as well as (A1) of Section~\ref{general approach}
with $\theta = \frac{2}{\tau}$, $C_1 = C_3$, $c_0 = \gamma$ and $C_2 =1$.

If in addition, a transitivity condition is satisfied,
then there is a unique conditionally invariant density
$\vf \in \B$ with eigenvalue $\lambda$.
\end{thma}

In order to eliminate periodicity and ensure transitivity for the map $T$
and for the tower, we
can impose the following transitivity condition.

\medskip
\noindent
\parbox{.1 \textwidth}{\bf(T1)}
\parbox[t]{.8 \textwidth}{Let $J$ be an interval of monotonicity for $T$.
There exists an $n_1 > 0$ such that $T^{n_1}J$ covers $I$ up to finitely many
points. }

\medskip
\noindent
Property (T1) is analogous to the covering property
for piecewise expanding maps of the interval without holes which is a
necessary and sufficient condition
for the existence of a unique absolutely continuous invariant measure
whose density is bounded away from zero (see \cite{liverani}).

Fix $\bar \alpha \geq - \log \beta/ \log \tau$ and if necessary, choose
$\beta$ closer to $1$ so that $\bar \alpha \leq \alpha$.
Let $I^\infty$ denote the set of points which never escape from $I$ and
define $\G = \{ \tf \in C^{\bar \alpha}(I): \tf>0 \mbox{ on } I^\infty \}$.
Denote by $\Lp_T$ the transfer operator of $T$ with respect to $\tm$ and
let $|\cdot |_1$ denote the $L^1(\tm)$-norm.
We prove the following theorem in Section~\ref{expanding}.

\begin{theorem}
\label{thm:exp convergence}
Let $T$ satisfy the hypotheses of Theorem A in
addition to condition (T1).

There exists $\lambda > 0$ such that for all $\tf \in \G$,
the escape rate with respect to $\teta = \tf \tm$
is well-defined and equal to $-\log \lambda$, i.e.,
\[
\lim_{n\to \infty} \frac{1}{n} \log \teta( I^n) = \log \lambda .
\]

There exists a unique \accim\ $\tmu$ with density $\tp$ and eigenvalue
$\lambda$ such that for all $\tf \in \G$, we have
\[
\left| \frac{\Lp_T^n \tf}{|\Lp_T^n \tf|_1} - \tp \right|_1
\leq C_T |\tf|_{C^{\bar \alpha}} \sigma^n
\]
for some $\sigma \in (\beta,1)$ and $C_T$ depends only on the smoothness
and distortion of the map $T$.
\end{theorem}

\begin{remark}
The results of Theorems A and \ref{thm:exp convergence}
generalize those obtained in \cite{chernov bedem} and \cite{liverani maume}
for $C^2$ expanding maps using bounded variation techniques.
We make no assumptions on the position of the holes, only their measure.
\end{remark}

%%%%%%%%%%%%%%%%%%%%%%%%%%

\subsubsection{Collet-Eckmann maps with singularities}
\label{CE results}

Collet-Eckmann maps are interval maps with critical points such that
the derivatives $D\T^n$ at the critical values increase exponentially.
We will follow the approach of \cite{DHL} which allows for
discontinuities and points with infinite derivative as in Lorenz maps.
(We try to use the same notation as in \cite{DHL}, but adding a ${}^*$
if there is a clash with our own notation.)
The map $\T:\I \to \I$ is locally $C^2$ and has a critical set
$\Crit = \Crit_c \cup \Crit_s$ consisting respectively of
genuine critical points $c$ with critical order $1 < \ell_c < \infty$
and singularities with critical order $0 < \ell_c \leq 1$.
At each of these points $\T$ is allowed to have a discontinuity as well,
so $c \in \Crit$ has a left and right critical order
which need not be the same.
Furthermore, $\T$ satisfies the following
conditions for all $\delta > 0$ (where $B_\delta(\Crit) = \cup_{c \in \sCrit} B_\delta(c)$ is a $\delta$-neighborhood of $\Crit$):

\begin{enumerate}
  \item[{\bf(C1)}]  {\em Expansion outside $B_\delta(\Crit)$: }  There exist
$\lambda^* > 0$ and $\kappa > 0$ such that for every $x$ and $n \geq 1$ such that
$x_0 = x, \dots, x_{n-1} = \T^{n-1}(x) \notin B_\delta(\Crit)$, we have\footnote{The addition of exponent $\ell_{\max}-1$ for $\delta$ in this formula is a correction to \cite[Condition (H1)]{DHL}, which affects the proofs only in the
sense that some constants will be different. The formula in \cite{DHL} cannot be realized for any $x$ at distance $\delta$ to any critical point $c \in \Crit_s$ with $\ell_c > 2$.}
\[
|D\T^n(x)| \geq \kappa \delta^{\ell_{\max}-1} e^{\lambda^* n},
\]
where $\ell_{\max} = \max\{ \ell_c : c \in \Crit_c\}$.
Moreover, if  $x_0 \in \T(B_\delta(\Crit))$ or $x_n \in B_\delta(\Crit)$,
then we have
\[
|D\T^n(x)| \geq \kappa e^{\lambda^* n}.
\]
\item[{\bf(C2)}]  {\em Slow recurrence and derivative growth along critical orbit: } There exists $\Lambda^* > 0$ such that for all $c \in \Crit_c$
there is $\alpha^*_c \in (0,\Lambda^*/(5\ell_c))$ such that\footnote{The fact that $\alpha^*_c$ depends on $c$ in this way is a correction to \cite[Condition (H2)]{DHL}, where this is not stated, but used in the proof of Lemma 2
of \cite{DHL}.}
\[
|D\T^k(\T(c))| \geq e^{\Lambda^* k}
\ \text{ and }\
\mbox{\textnormal{dist}}(\T^k(c), \Crit) > \delta e^{-\alpha^*_c k}
\quad \text{ for all } k \geq 1.
\]
\item[{\bf(C3)}]  {\em Density of preimages: }  There exists $c^* \in \Crit$
whose preimages are dense in $\I$, and no other critical point is among these preimages.
\end{enumerate}
Condition (C1) follows for piecewise $C^2$ maps from Ma\~n\'e's Theorem,
 see \cite[Chapter III.5]{dMvS}. The first half of condition
(C2) is the actual Collet-Eckmann condition, and the second half is a slow
recurrence condition. 
Condition (C3) excludes the existence of non-repelling periodic points.

In \cite{DHL}, $\alpha_c^*$ is assumed to be small relative to $\lambda^*$
and $\Lambda^*$.  We keep the same restriction on $\alpha_c^*$ and
do not need to shrink it further after the introduction of holes.

\bigskip
\noindent
{\bf Introduction of Holes.}
In order to apply the tower construction of \cite{DHL} to our setting,
we place several conditions
on the placement of the holes in the interval $\I$.
We adopt notation similar to that in Section~\ref{exp results}.

A hole $\tH$ in $\I$ is a finite union of open intervals $\tH_j$,
$j = 1, \ldots, L$.
Let $I = \I \backslash \tH$ and set $I^n = \bigcap_{i=0}^n \T^{-i}I$.
Define $T = \T|I^1$ and let $\tm$ denote Lebesgue measure on $\I$.

\medskip
\noindent
\parbox{.1 \textwidth}{\bf(B1)}
\parbox[t]{.8 \textwidth}{Let $\alpha^*_c>0$ be as in (C2).
For all $c \in \Crit_c$ and $k \geq 0$,
\[
\mbox{\textnormal{dist}}(\T^k(c), \partial \tH) > \delta e^{-\alpha^*_c k}.
\]
}

\medskip
\noindent
Our second condition on $\tH$ is that the positions of its connected
components are generic with respect to
one another.  This condition will also double as a transitivity condition
on the constructed tower
which ensures our conditionally invariant density will be bounded away from
zero.  In order to
formulate this condition, we need the following fact about
$C^2$ nonflat nonrenormalizable maps satisfying (C1)-(C2).
\begin{equation}
\label{eq:no holes mixing}
\begin{array}{l}
\text{For all $\delta>0$ there exists $n = n(\delta)$ such that for all
intervals $\omega \subseteq \I$ with }
|\omega| \geq \frac{\delta}{3}, \\[0.1cm]
\quad\ (i)\ \T^n\omega \supseteq \I, \text{ and }  \\[0.1cm]
\quad (ii)\ \text{there is a subinterval $\omega' \subset \omega$ such that
$\T^{n'}$ maps $\omega'$ diffeomorphically }\\[0.01cm]
\quad \quad\ \text{ onto }(c^*-3\delta, c^*+3\delta) \text{ for some } 0 < n' \leq n.
\end{array}
\end{equation}
We also need some genericity conditions on the placement of the components
of the hole.
Within each component $\tH_j$, we place an artificial critical point
$b_j$, so $\Crit_{\mbox{\tiny hole}} = \{ b_1, \dots, b_L \}$.
The points $b_j$ are positioned so that the following holds:
\medskip

\noindent
\parbox{.07 \textwidth}{\bf(B2)}
\parbox[t]{.9 \textwidth}{(a) $\text{orb}(b_j) \cap c = \emptyset$ for all
$1 \leq j \leq L$ and $c \in \Crit_{\mbox{\tiny hole}} \cup \Crit_c$. \\ [0.1cm]
(b) Let $\T^{-1} (\T b_j) = \cup_{i=1}^{K_j} g_{j,i}$.
For all $j,k \in \{1, \ldots, L\}$, there exists $i \in \{1, \ldots, K_j \}$
such that
$\T^\ell b_k \neq g_{j,i}$ for
$1 \leq \ell \leq n(\delta)$.}
\medskip

\noindent
Here $n(\delta)$ is the integer corresponding to $\delta$ in
\eqref{eq:no holes mixing} and $\delta$ is chosen so small
that: (i) all points in $\Crit_c \cup \Crit_s \cup \Crit_{hole}$
are at least $\delta$ apart, and (ii) for each $j = 1, \dots, L$,
there is $\tau = \tau(j) \geq 1$ such that
\begin{equation}\label{eq:tau}
|D\T^{\tau}(x)| \geq \max\{ \kappa e^{\lambda^* \tau}, 6 \}
\quad \text{ for all } x \in B_\delta(b_j).
\end{equation}
Condition (C1) implies that
$|D\T^{\tau}(x)| \geq \kappa e^{\lambda^* \tau}$ whenever
$x \notin B_\delta(\Crit_c)$ and
$T^{\tau}(x) \in B_\delta(\Crit_c)$, so by taking $\delta$ small,
and using assumption (B2)(a), we can indeed find $\tau$ such that
also $|D\T^{\tau}(x)| \geq 6$.

As before, let $I^\infty$ denote the set of points which never escapes from $I$ and
let $\G = \{ \tf \in C^{\bar \alpha}(I): \tf>0 \mbox{ on } I^\infty \}$.
We prove the following theorem in Section~\ref{logistic}.

\begin{theorem}
\label{thm:CE convergence}
Let $\T$ be a nonrenormalizable map satisfying conditions (C1)-(C3) and let $\tH$ be
a sufficiently small hole satisfying (B1)-(B2).

There exists $\lambda > 0$ such that for all $\tf \in \G$,
the escape rate with respect to $\teta = \tf \tm$
is well-defined and equal to $-\log \lambda$, i.e.,
\[
\lim_{n\to \infty} \frac{1}{n} \log \teta( I^n) = \log \lambda .
\]

Moreover, there exists a unique \accim\ $\tmu$ with density $\tp$ and eigenvalue
$\lambda$ such that for all $\tf \in \G$, we have
\[
\left| \frac{\Lp_T^n \tf}{|\Lp_T^n \tf|_1} - \tp \right|_1
\leq C_T |\tf|_{C^{\bar \alpha}} \tilde \sigma^n
\]
for some $\tilde \sigma \in (\beta,1)$ and $C_T$ depends only on the smoothness
and distortion of the map $T$.
\end{theorem}

\begin{remark}
When $\T$ is a Misiurewicz map (i.e., all critical points are
nonrecurrent and all periodic points are non-repelling),
it is possible to give a constructive bound on the
size of the hole in terms of explicit constants.
In this case we assume that the critical point does not fall into the hole,
but there is no need for condition (B2)(a),
see \cite[Section 2.2]{demers logistic}.
\end{remark}

\noindent
{\bf Small Hole Limit.}
Fix $L$ distinct points $b_1, \ldots, b_L \in \I$ which we consider to be
infinitesimal holes satisfying (B1) and (B2).
We call this hole of measure zero $\tH^{(0)}$ with
components $\tH_j^{(0)} = b_j$, $j = 1, \ldots, L$.
For each $h>0$, we
then define a
family of holes $\Ho(h)$ such that $\tH \in \Ho(h)$ if and only if
\begin{enumerate}
\item $b_j \in \tH_j$ and
$\tm(\tH_j) \leq h$ for each $1 \leq j \leq L$;
\item $\tH$ satisfies (B1).
\end{enumerate}
When we shrink a hole in $\Ho(h)$, we keep $b_1, \ldots, b_L$ fixed
and simply choose a smaller $h$.  The following theorem is proved in
Section~\ref{small hole}.

\begin{theorem}
\label{thm:small hole}
Let $\T$ satisfy the hypotheses of Theorem~\ref{thm:CE convergence} and
let $\tH^{(h)} \in \Ho(h)$ be a family of holes.
Let $d\tmu_h = \tp_h d\tm$ be the \accim\ given by Theorem~\ref{thm:CE convergence} with eigenvalue $\lambda_h$.  Then
$\lambda_h \to 1$ and $\tmu_h$ converges weakly to the unique SRB measure
for $\T$ as $h \to 0$.
\end{theorem}

\begin{remark}
A similar theorem was proved for piecewise expanding maps in \cite{demers exp}
and for Misiurewicz maps in \cite{demers logistic}.
\end{remark}

%%%%%%%%%%%%%%%%%%%%%%%%%%%%%%%%%%%%%%

\subsection{An Equilibrium Principle for $(T,X)$}
\label{T-equilibrium results}

In this section, $T$ is either an expanding map satisfying the
assumptions of Theorem~\ref{thm:exp convergence} or a Collet-Eckmann map
with singularities
satisfying the assumptions of Theorem~\ref{thm:CE convergence}.

Recall the invariant measure $\nu$ supported on $\Delta^\infty$ introduced
in Proposition~\ref{prop:inv measure}.  The measure $\tnu := \pi_*\nu$
is $T$-invariant and is supported on $X^\infty = \pi(\Delta^\infty)$.
We show that $\tnu$ is physically relevant to
the system $(T,X)$ in two ways.

\begin{theorem}
\label{thm:inverse limit}
The invariant measure $\tnu$ is characterized by
\[
\tnu(\tf) = \lim_{n\to\infty} \lambda^{-n} \int_{X^n} \tf d\tmu
\]
for all functions $\tf \in C^{\bar \alpha}(X)$.  In addition,
$\tnu$ is ergodic and enjoys exponential decay of correlations on
H\"{o}lder observables.
\end{theorem}

Although $\tnu$ is defined simply as $\pi_*\nu$, the preceding theorem
gives a characterization of $\tnu$ which is independent of the tower
construction.  This is important for two reasons: first,
it implies that two different tower constructions will yield
the same invariant measure; second, it eliminates the
need to construct a tower in order to compute $\tnu$.

The second theorem is a consequence of Theorem~\ref{thm:F-equilibrium}.

\begin{theorem}
\label{thm:T-equilibrium}
Let $\M'_T = \pi_*\M_F = \{ \pi_*\eta: \eta \in \M_F \}$ be the set of
$T$-invariant Borel probability measures on $X$ whose lift to $\Delta$
is well-defined.
Then
\[
\log \lambda = \sup_{\eta \in \M'_T} \left\{ h_\eta(T)
    - \int_X \log JT \, d\eta \right\}
\]
where $JT$ is the Jacobian of $T$ with respect to $\tm$.  The invariant
measure $\tnu$ is the unique nonsingular
measure $\tilde{\eta}$ in $\M'_T$ which
attains the supremum.
\end{theorem}

\begin{remark}
As stated in Theorem~\ref{thm:T-equilibrium}, the equilibrium principle
applies to the collection of invariant measures $\M'_T$ whose lift to 
$\Delta$ is well-defined.
So a priori, $\tilde \nu$ need not be the global equilibrium state in
$\M_T$, the set of $T$-invariant
measures supported on 
$Y^\infty = \{ x \in X : \T^n(x) \notin \tH \mbox{ for all } n \geq 0\}$.
(As $X^\infty = \pi(\Delta^\infty)$, $Y^\infty \supset X^\infty$.)

For $C^2$ Collet-Eckmann maps without singularities, however,
the equilibrium state $\tnu$ is indeed global.  This is due to
\cite[Theorem 4]{young ld} which in our setting guarantees
$\log \lambda \geq \sup_{\eta \in \M_T} \left\{ h_\eta(T)
- \int_X \log JT \, d\eta \right\}$.  Since $\tnu \in \M'_T \subseteq \M_T$ attains
the supremum, the inequality is in fact an equality.

\iffalse
This issue was addressed in \cite{bruin todd} for multimodal interval maps
without holes, using the Hofbauer tower.

For unimodal Misiurewicz maps, an easier argument is possible.
First observe that a $\T$-invariant probability measure $\tilde\eta$
on $Y^\infty$ has a lift to $\Delta$ whenever $\tilde\eta(\Lambda) > 0$ where
$\Lambda \subset X$ is the reference set identified with the base of the tower.
If  $\tilde\eta(\Lambda) = 0$, then $\tilde\mu$ is supported on a
hyperbolic set $C$ of Hausdorff dimension $d < 1$ and positive lower
Lyapunov exponent
$\underline\lambda :=
\inf_{x \in C} \liminf_{n \to \infty} \frac1n \log |(\T^n)'(x)|$.
It follows by well-known entropy formulas, see e.g.
\cite[Theorem 21.3]{pesin} that the dimension of $\tilde\eta$
is $h_{\tilde \eta}\ /\int \log |\T'| \ d\tilde\eta \leq d < 1$,
provided $h_{\tilde \eta} > 0$. Therefore,
regardless of whether $h_{\tilde \eta} > 0$ or not,
so $h_{\tilde\mu} - \int \log|\T'| \ d\tilde\mu \leq (d-1)
\underline\lambda < 0$.
For sufficiently small holes,
$(d-1) \underline\lambda < \log \lambda \leq 0$, so
the equilibrium state $\tilde\nu$ is indeed global.
\fi
\end{remark}

Theorem~\ref{thm:inverse limit} is proved in Section~\ref{inverse limit}
and Theorem~\ref{thm:T-equilibrium} is proved in Section~\ref{T-equilibrium}.

%%%%%%%%%%%%%%%%%%%%%%%%%%%%%%%%%%%%%%
%%%%%%%%%%%%%%%%%%%%%%%%%%%%%%%%%%%%%%

\section{Convergence Properties of $\Lp$}
\label{quasi-compact}

\subsection{Lasota-Yorke inequalities}
\label{lasota-yorke}

In this section we prove Proposition~\ref{prop:lasota-yorke} by deriving
Lasota-Yorke type inequalities for $\| \cdot \|_\infty$ and
$\| \cdot \|_{\mbox{\tiny Lip}}$,
\[
\|\Lp^nf \|_\infty \leq C \beta^n \|f\|_{\mbox{\tiny Lip}} + C |f|_1,
\qquad
\|\Lp^nf\|_{\mbox{\tiny Lip}} \leq C \beta^n \|f\|_{\mbox{\tiny Lip}} + C|f|_1.
\]

\medskip
\noindent
\begin{proof}[Proof of Proposition~\ref{prop:lasota-yorke}]
We fix $n \in \N$ and separate the estimates into two
parts: those for $\zlj$ with $\ell \geq n$ and
those with $\ell < n$.

\medskip
\noindent
{\bf Estimate \# 1}.
For any $x \in \zlj$ with $\ell \geq n$ and $f \in \B$, note that $\Lp^nf(x) = f(F^{-n}x)$ since $g_n(F^{-n}x)=1$.
This allows us to estimate,
\[
\|\Lp^nf_{\ell, j}\|_\infty \; := \; |f(F^{-n})_{\ell, j}|_\infty \beta^\ell
   \; = \; (|f_{\ell-n,j}|_\infty \beta^{\ell-n}) \beta^n
   \; = \; \| f_{\ell-n,j} \|_\infty \beta^n.
\]

\medskip
\noindent
{\bf Estimate \#2}.
Again choose any $\zlj$ with $\ell \geq n$.  Then
\[
\begin{split}
\| \Lp^nf_{\ell, j}\|_{\mbox{\tiny Lip}} \;
& := \; \sup_{x,y \in \zlj} \frac{|f(F^{-n}x) - f(F^{-n}y)|}{d(x,y)}
    \beta^\ell \\
& = \;  \beta^n \sup_{x,y \in \zlj} \frac{|f(F^{-n}x) - f(F^{-n}y)|}
    {d(F^{-n}x, F^{-n}y) \beta^{-n}} \beta^{\ell - n}
 \; = \; \beta^{2n} \| f_{\ell - n, j} \|_{\mbox{\tiny Lip}}
\end{split}
\]
since $s(x,y) = s(F^{-n}x, F^{-n}y)-n$.

\medskip
\noindent
{\bf Estimate \#3}.
Let $x \in Z_{0,j}$ be a point in $\Delta_0$.  We denote by $\E_n$ the cylinder sets of length $n$ with respect
to the partition $\Z$ and let $E_n(y)$ denote the element of $\E_n$ containing $y$.
\[
|\Lp^nf(x)| \leq \sum_{y \in F^{-n}x} |f(y)| g_n(y)
   \leq \sum_{y \in F^{-n}x } |f(\ab)|g_n(y)
     + g_n(y) \frac{|f(y) - f(\ab)|}{d(y,\ab)} d(y,\ab)
\]
where $\ab \in E_n(y)$ is any point satisfying
$|f(\ab)| \leq \frac{1}{m(E_n(y))} \int_{E_n(y)} |f| dm$.
By \eqref{eq:distortion2},
\begin{equation}
\label{eq:g_n dist}
g_n(y) \leq (C_1+1)c_0^{-1}m(E_n(y)).
\end{equation}
Finally, note that
$F^ny, F^n\ab \in Z_{0,j}$ so that $n$ is a return time for $y$ and $\ab$.
If $y \in \Delta_{\ell(y)}$, then the definition of $d$ from
Section~\ref{holes} implies
$d(y,\ab) \leq \beta^n$.

Putting this together with \eqref{eq:g_n dist},
we estimate
\begin{eqnarray}
|\Lp^nf(x)|
    & \leq & \sum_{y \in F^{-n}x } \frac{g_n(y)}{m(E_n(y))} \int_{E_n(y)}|f|dm
    + g_n(y) \beta^{-\ell(y)} \|f\|_{\mbox{\tiny Lip}} \beta^n
               \nonumber \\
    & \leq & \sum_{y \in F^{-n}x } (1+C_1)c_0^{-1} \int_{E_n(y)}|f|dm
    + (1+C_1)c_0^{-1} m(E_n(y)) \beta^{n-\ell(y)} \|f\|_{\mbox{\tiny Lip}}  \; \; \;
                   \label{eq:infty split}  \\
    & \leq & (1+C_1)c_0^{-1} \int_{\Delta^n} |f|\, dm
    +  C \beta^n \|f\|_{\mbox{\tiny Lip}}           \nonumber
\end{eqnarray}
where the second sum is finite since $\beta > \theta$.

\medskip
\noindent
{\bf Estimate \#4}.
Let $x, y \in Z_{0,j}$, and let $x' \in F^{-n}x$, $y' \in F^{-n}y$,
denote preimages taken along the same branch of $F^{-n}$.
Then summing over all inverse branches gives
\begin{eqnarray}
\frac{|\Lp^nf(x) - \Lp^nf(y)|}{d(x,y)} & \leq &
\sum_{x' \in F^{-n}x }
  \frac{|g_n(x')f(x') - g_n(y')f(y')|}{d(x,y)}               \nonumber      \\
   & \leq & \sum_{x' \in F^{-n}x }      g_n(x') \frac{|f(x') - f(y')|}{d(x,y)}
     + |f(y')| \frac{|g_n(x') - g_n(y')|}{d(x,y)}
                       \label{eq:reg split}           \\
   & \leq & \sum_{x'\in F^{-n}x}  (1+C_1)c_0^{-1} m(E_n(x'))
      \|f\|_{\mbox{\tiny Lip}} \beta^{n-\ell(x')}
     + |f(y')| C_1 g_n(y')           \nonumber
\end{eqnarray}
where we have used \eqref{eq:distortion} and \eqref{eq:g_n dist} for $g_n$ in the last
line as well as the fact
that $d(x',y') = \beta^n d(x,y)$ for $x' \in \Delta_{\ell(x')}$.
The second sum is identical to that in Estimate \#3.  Thus
\[
\frac{|\Lp^nf(x) - \Lp^nf(y)|}{d(x,y)} \leq C \beta^n \|f\|_{\mbox{\tiny Lip}} + C |f|_1.
\]

\medskip
Now on $\zlj$ with $\ell < n$, we can combine Estimates \#1 and \#3 to obtain
\[
\| \Lp^nf_{\ell, j}\|_\infty \leq \beta^\ell \| \Lp^{n-\ell}f_{0,j} \|_\infty
 \leq \beta^\ell ( C \beta^{n-\ell} \|f\|_{\mbox{\tiny Lip}} + C |f|_1 )
\]
which implies the estimate for the $\| \cdot \|_\infty$-norm.

Similarly, we can combine Estimates \#2 and \#4 to obtain
\[
\| \Lp^nf_{\ell, j}\|_{\mbox{\tiny Lip}} \leq \beta^{2\ell}
    \| \Lp^{n-\ell}f_{0,j} \|_{\mbox{\tiny Lip}}
   \leq \beta^{2\ell} (C \beta^{n-\ell} \|f\|_{\mbox{\tiny Lip}} + C |f|_1 )
\]
which completes the estimate for the $\| \cdot \|_{\mbox{\tiny Lip}}$-norm.
\end{proof}

%%%%%%%%%%%%%%%%%%%%%%%%%%%%%%%%%%%%%%%%%%%%%%%%%%%%%%%%%%%

\subsection{Spectral Gap}
\label{spectral gap}

Although Proposition~\ref{prop:lasota-yorke} implies that the essential
spectral radius of $\Lp$ on $\B$ is less than or equal to $\beta$, we must
still ensure that there is a spectral gap, i.e.,\ that there is
an eigenvalue $\beta < \lambda < 1$ which represents the rate of escape of
typical elements of $\B$.

This fact follows from the bound on the measure of the hole $H$ given by
(H1).  To prove it, it will be convenient to recall some results from
\cite{demers exp}
which concern the nonlinear operator $\Lp_1f := \Lp f/|\Lp f|_1$.
Thus $\Lp_1f$ represents the normalized
push-forward density which is conditioned on non-absorption by the hole.
In \cite{demers exp} it was shown that for small holes,
$\Lp_1$ preserves a convex subset of $\B$ defined by
\[
\B_M = \{ f \in \B : f \geq 0, |f|_1 = 1,  \|f\|_\infty \leq M, \|f\|_{\mbox{\tiny log}} \leq M \}
\]
where
\[
\| f \|_{\mbox{\tiny log}}
   = \sup_{\ell,j} \mbox{Lip}(\log f_{\ell,j}).
\]

We include the proof of the proposition here for clarity and also to
formulate the bound on $H$ in terms of the present notation.

\begin{proposition}
\label{prop:regularity}
Let $M \in \left( (1+C_1)c_0^{-1}, \, (1-\beta)q^{-1} \right)$.  Then
\begin{enumerate}
  \item[(i)] $\Lp_1^n$ maps $\B_M$ into itself for $n$ sufficiently large.
  \item[(ii)] There exists $\beta' > \beta$ such that $|\Lp f|_1 \geq \beta'$
    for all $f \in \B_M$.
\end{enumerate}
\end{proposition}

\begin{proof}
Note that $\mbox{Lip}(\log f_{\ell,j})$ is equivalent to
$\displaystyle
\sup_{x,y \in Z_{\ell,j}} \frac{|f(x) - f(y)|}{f(x)d(x,y)}$.

We will work with this expression in the following estimates.  For $f \in \B_M$, we prove the analogue of
Estimates \#1-\#4 from Section~\ref{lasota-yorke}
using $\| \cdot \|_{\mbox{\tiny log}}$.  Estimates \#1 and \#2 are the same so we do not repeat those.

To prove Estimate \#3, note that equation \eqref{eq:infty split} becomes
\begin{equation}
\label{eq:infty new split}
\begin{split}
|\Lp^nf(x)| & \leq \sum_{y \in F^{-n}x} f(\ab) g_n(y) + g_n(y) \frac{|f(y) - f(\ab)|}{f(\ab)d(y,\ab)}d(y,\ab) f(\ab) \\
            & \leq \sum_{y \in F^{-n}x} \frac{g_n(y)}{m(E_n(y))} \int_{E_n(y)} f dm
            + \frac{g_n(y)}{m(E_n(y))} \|f\|_{\mbox{\tiny log}} \beta^n \int_{E_n(y)} f dm \\
            & \leq (1+C_1)c_0^{-1} (1 + \beta^n \|f\|_{\mbox{\tiny log}}) \int_{\Delta^n} f dm
\end{split}
\end{equation}
where in the last line we have used the fact that $\sum_{y \in F^{-n}x} \int_{E_n(y)} f dm \leq \int_{\Delta^n} f dm$.

To modify Estimate \# 4, we need the following fact:
If $\sum_i a_i$ and $\sum_i b_i$ are two series of positive
terms, then
$\displaystyle
\frac{\sum_i a_i}{\sum_i b_i} \leq \sup_i \frac{a_i}{b_i}$.
Equation \eqref{eq:reg split} becomes
\begin{equation}
\label{eq:log reg split}
\begin{split}
\frac{|\Lp^nf(x) - \Lp^nf(y)|}{d(x,y) \Lp^nf(x)} & \leq
                \frac{\sum_{x' \in F^{-n}x} g_n(y') \frac{|f(x')-f(y')|}{d(x,y)}}{\sum_{x' \in F^{-n}x} g_n(x')f(x')}
                + \frac{\sum_{x' \in F^{-n}x} f(x') \frac{|g_n(x')-g_n(y')|}{d(x,y)}}{\sum_{x' \in F^{-n}x} g_n(x')f(x')} \\
    & \leq \sup_{x' \in F^{-n}x} \frac{g_n(y')}{g_n(x')} \frac{|f(x')-f(y')|}{d(x,y)f(x')}
                + \sup_{x' \in F^{-n}x} \frac{\left|1-\frac{g_n(y')}{g_n(x')}\right|}{d(x,y)}  \\
    & \leq (1+C_1) \beta^n \|f\|_{\mbox{\tiny log}} + C_1.
\end{split}
\end{equation}

Since $\| \cdot \|_{\mbox{\tiny log}}$ is scale invariant,
\eqref{eq:log reg split} implies for all $n \geq 0$,
\begin{equation}
\label{eq:log reg}
\|\Lp_1^n f\|_{\mbox{\tiny log}} = \| \Lp^n f\|_{\mbox{\tiny log}}
\leq (1+C_1) \beta^n \|f\|_{\mbox{\tiny log}} + C_1.
\end{equation}

Using \eqref{eq:infty new split},
\[
\frac{|\Lp^n f(x)|}{|\Lp^n f|_1} \leq \frac{(1+C_1)c_0^{-1} (1 + \beta^n \|f\|_{\mbox{\tiny log}}) \int_{\Delta^n} f dm}
                                           {\int_{\Delta^n} f dm}
                                  \leq (1+C_1)c_0^{-1} (1 + \beta^n \|f\|_{\mbox{\tiny log}})
\]
so that the $\| \cdot \|_\infty$-norm stays bounded on the base of the tower.
In order for this norm
to remain bounded on successive levels, we need to ensure that
$|\Lp f|_1 \geq \beta$ for each $f \in \B_M$.   Compute that
\begin{eqnarray*}
\int_\Delta \Lp f \, dm & = & \int_{\Delta^1} f \, dm
    \; =  \; 1 - \sum_{\ell \geq 1} \int_{\F^{-1}H_{\ell,j}}f \,dm  \\
    & \geq & 1 - \sum_{\ell \geq 1} \|f_{\ell-1,j}\|_\infty \beta^{-(\ell-1)}
       m(H_{\ell})
    \; \geq \; 1 - M \sum_{\ell \geq 1} \beta^{-(\ell-1)} m(H_{\ell}).
\end{eqnarray*}

Recall that
$q = \sum_{\ell \geq 1} \beta^{-(\ell-1)} m(H_{\ell})$.
Thus $|\Lp f|_1 > \beta$ if $1 - qM > \beta$ and $M$ must be chosen large
enough so that $\Lp_1^n$ maps
$\B_M$ back into itself for large enough $n$.
Equations~\eqref{eq:infty new split} and \eqref{eq:log reg split}
require that we choose $M \in \left( (1+C_1)c_0^{-1}, (1-\beta)q^{-1} \right)$.
Thus $q < \frac{1-\beta}{M} < \frac{(1-\beta)c_0}{1+C_1}$ is a sufficient
condition on the size of $H$ and is precisely assumption (H1).
\end{proof}

\medskip

\begin{proof}[Proof of
Proposition~\ref{prop:spectral picture}.]
The proof divides into several steps.

\medskip
{\em 1. Quasi-compactness of $\Lp$.}
Proposition~\ref{prop:regularity} implies that there exists $N\ge1$ such that
$\Lp_1^N \B_M \subset \B_M$.
Since $\Lp_1^N$ is continuous on $\B_M$, which is a convex, compact
subset of $L^1(m)$,
the Schauder-Tychonoff theorem guarantees the existence of a fixed
point $\vf \in \B_M$, which is a conditionally invariant density for
$\Lp^N$ with eigenvalue $\rho = \int_{\Delta^N} \vf\, dm$.

Proposition~\ref{prop:lasota-yorke} implies
that the essential spectral radius of $\Lp^N$ is bounded by $\beta^N$
and Proposition~\ref{prop:regularity} guarantees that
$\rho > \beta^N$.

Thus $\Lp^N$ is quasi-compact with spectral radius at least $\rho$.
We conclude that $\Lp$ is quasi-compact with spectral radius at least
$\lambda := \rho^{1/N}$ and essential spectral radius $\beta < \lambda$.

Let $N_0\leq N$ be the least positive integer such that
$\Lp^{N_0}\vf=\lambda^{N_0}\vf$.
In the next part of the proof, Steps 2--5, we assume that $F$ is mixing and
that $N_0=1$.   These assumptions are removed in Steps 6 and 7.

\medskip
{\em 2. The density $\vf$.}
We claim that there exists $\delta>0$ such that
$\delta \leq \vf|\Delta_0 \leq \delta^{-1}$.   It is then immediate from the
conditional invariance condition $\lambda^{-1}\Lp\vf=\vf$ that
$\delta\lambda^{-\ell} \leq \vf|\Delta_\ell \leq \delta^{-1}\lambda^{-\ell}$.

By conditional invariance, for $x \in \Delta_\ell$,
$\vf(x) = \lambda^{-\ell} \vf(F^{-\ell} x)$, so that $\vf \equiv 0$ on
$\Delta$ if $\vf \equiv 0$ on $\Delta_0$.  Thus there exists
$x \in \Delta_0$ such that $\vf(x)>0$.  Using conditional invariance
once more, we obtain $x' \in F^{-1}x$ such that $\vf(x') > 0$.
Let $Z$ be the partition element containing $x'$.
Since $\vf\in\B_M$, it follows that $\vf\ge\kappa>0$ on $Z$.
By construction, $F(Z) \supseteq Z'$ for some $Z' \in \Z_0^{im}$.
By conditional invariance,
$\inf_{Z'}\vf \geq \lambda^{-1} \kappa \inf_Z g>0$.
By transitivity, conditional invariance, and the property that $\vf\in\B_M$,
we obtain a similar lower bound for each $Z'\in\Z_0^{im}$.
The claim follows from finiteness of the partition $\Z_0^{im}$.

\medskip
{\em 3. Spectral radius.}
Now suppose $f \in \B$ such that $\Lp f = \alpha f$ and $|\alpha| > \lambda$.
Note that $f$ satisfies
$f(x) = \alpha^{-\ell} f(F^{-\ell}x)$ for each $x \in \Delta_\ell$,
$\ell \geq 1$.
Since $\vf \geq \delta$, there exists $K>0$ such that
$K\vf \geq |f|$ on $\Delta_0$.  But since $f$ grows like $\alpha^{-\ell}$
and $\vf$ grows like $\lambda^{-\ell}$ on level $\ell$, we have
$K \vf \geq |f|$ on $\Delta$.  By the positivity of $\Lp$,
$K \Lp^n \vf \geq \Lp^n |f| \geq | \Lp^n f|$ for each $n$.  But this
implies that $K \lambda^n \vf \geq |\alpha|^n|f|$ for each $n$.
Since $\lambda < |\alpha|$, it follows that $f\equiv0$.
Hence $\Lp$ has spectral radius precisely $\lambda$.

\medskip
{\em 4. Simplicity of $\lambda$.}
Suppose $f\in\B$ such that $\Lp f = \lambda f$.  As in Step~3, we can choose $K>0$ such that $f+K\varphi>0$.   Let
$\psi = (f + K \vf)/C > 0$ where $C=|f|_1+K$ is chosen so that
$\psi$ is a probability
density.  Define $\psi_s = s \vf + (1-s) \psi$ and let $J = \{ s \in \R :
\inf_\Delta \psi_s > 0 \}$.
Note that for $s \in J$, $\Lp \psi_s = \lambda \psi_s$ and $| \psi_s |_1 = 1$.
Since $\psi_s$ is Lipschitz and bounded away from zero,
$\| \psi_s \|_{\mbox{\tiny log}} < \infty$.  In fact, \eqref{eq:log reg}
implies that
$\| \psi_s \|_{\mbox{\tiny log}}
= \lim_{n \to \infty} \| \Lp_1^n \psi_s \|_{\mbox{\tiny log}} \leq M$,
so that $\psi_s \in \B_M$ for all $s \in J$.

Since $\psi_s$ is conditionally invariant, the identity
$\psi_s |_{\Delta_\ell} = \lambda^{-\ell} \psi_s |_{\Delta_0}$ implies that
$\inf_\Delta \psi_s = \inf_{\Delta_0} \psi_s$, so that $J$ is open.
Now let $t \in \partial J$.  Since $\psi_s \in \B_M$ for all $s \in J$
and $\B_M$ is closed, we have $\psi_t \in \B_M$.  If $\psi_t$ vanishes
on $\Delta_0$, then $\psi_t$ vanishes on an entire element $Z' \in \Z_0^{im}$.
Since $\psi_t \geq 0$, this implies that $\psi_t \equiv 0$ on all
elements of $\Z_0^{im}$ which map to $Z$ and by transitivity $\psi_t$ is zero on all
of $\Delta$.
Thus $\psi_t$ has strictly positive infimum on $\Delta_0$ and since it is
conditionally invariant, it must have the same infimum on $\Delta$.
Thus $t \in J$, so $J$ is closed.  Since $J$ is nonempty, $J=\R$, which is
only possible if $f = c \vf$ for some $c \in \R$.

It remains to eliminate generalized eigenvectors.  Suppose $f \in \B$
such that $\Lp f = \lambda (f + \vf)$.  Then
$\Lp^n f = \lambda^n f + n \lambda^n \vf
= \lambda^n f + \Lp^n(n \vf)$ so that
$\Lp^n(f - n \vf) = \lambda^n f$ .
This implies that for $x \in \Delta_\ell$,
\[
f(x) = \lambda^{-\ell} (f-\ell \vf)\circ F^{-\ell}(x) .
\]
Since for $\ell$ large enough, $f-\ell \vf < 0$ on $\Delta_0$, we have
$f<0$ on $\cup_{\ell \geq L}\Delta_\ell$ for some $L>0$.

Choose $K>0$ large such that
that $\psi := f - K\vf < 0$ on $\cup_{\ell < L} \Delta_\ell$.
Since $\psi<f$, we have $\psi<0$ on the whole of $\Delta$.
For each $n \geq 0$,
\[
0 \;  > \; \lambda^{-n} \int_{\Delta^n} \psi \, dm
\; = \; \lambda^{-n} \int_\Delta \Lp^n \psi \, dm
      = \; \int_\Delta (\psi + n \vf) \, dm
         \; = \; \int_\Delta \psi \, dm + n,
\]
which is a contradiction.

\medskip
{\em 5. Absence of peripheral spectrum.} Suppose $f \in \B$, $|f|_1=1$,
such that $\Lp f = \alpha f$, where $\alpha = \lambda e^{i \omega}$,
$\omega\in(0,2\pi)$.
We follow an approach similar to Step~4, modified to take
into account the fact that $f$ is complex and $\alpha \neq \lambda$.
Notice that by conditional invariance,
\begin{equation}
\label{eq:rotation}
f |_{\Delta_\ell} = \lambda^{-\ell} e^{-i \omega \ell} f |_{\Delta_0},
\end{equation}
so that $f$ grows like $\vf$ plus a rotation up the levels of the tower.

Define $\psi = (\re(f) + K \vf)/C'$, where $K$ is chosen
large enough that $\psi > 0$ and $C'$ normalizes $|\psi|_1 =1$.
By replacing $f$ with $-f$ if necessary, we can guarantee that
$\int_{\Delta} \re(f) dm \leq 0$,
so that $C' \leq K$.  Also notice that since $f$ and $\vf$ grow at the
same rate, there exists $\delta_0>0$ such that
\begin{equation}
\label{eq:grow}
\delta_0 \lambda^{-\ell} \leq \psi(x) \leq \delta_0^{-1} \lambda^{-\ell}
\end{equation}
for $x \in \Delta_\ell$.

As before, define $\psi_s = s \vf + (1-s) \psi$ and let $J = \{ s \in \R :
\inf_\Delta \psi_s > 0 \}$.  Due to \eqref{eq:grow}, $J$ is open.  However,
$\psi_s$ is not conditionally invariant since $\alpha \neq \lambda$ so the
second part of the argument needs some modification.

Notice that
\begin{equation}
\label{eq:normalize}
\lambda^{-n} \Lp^n \psi = (\re(e^{i \omega n} f) + K\vf)/C'
\end{equation}
so we may choose a sequence $n_k$ such that
$\lambda^{-n_k}\Lp^{n_k} \psi \to \psi$ as $k \to \infty$.
This implies also that $\lambda^{-n_k}\Lp^{n_k} \psi_s \to \psi_s$ along
the same sequence and by \eqref{eq:log reg} we have
$\| \psi_s \|_{\mbox{\tiny log}} \leq M$ so that $\psi_s \in \B_M$ for
$s \in J$.  Now let $t > 1$ be the right endpoint of $J$.  Since
$\B_M$ is closed, we have $\psi_t \in \B_M$.  The rest of Step~5
relies on the following lemma.

\begin{lemma}
\label{lem:positive}
$\psi_t$ is bounded away from $0$.
\end{lemma}

It is easy to see that the lemma completes the proof of Step~5
since then $t \in J$ and we conclude that  $J \supset \R^+$.
Now $\psi_s >0$ for all $s >0$, implies $\vf > \psi$.  Thus
\[
\vf > (\re(f) + K \vf)/C' \; \Rightarrow \; (C' - K)\vf > \re(f)
\; \Rightarrow \; 0 >  \re(f)
\]
since $C' \leq K$.  But $\re(f)$ must change sign on $\Delta$
due to the rotation as we move up the levels of the tower given by
\eqref{eq:rotation}.  This contradicts the existence of $\alpha$.

\begin{proof}[Proof of Lemma~\ref{lem:positive}.]
Since $|\psi_t|_1 =1$ and $\psi_t \geq 0$,
there exists $\ell\ge0$ and $x\in\Delta_\ell$ such that $\psi_t(x)>0$.
Since $\lambda^{-n_k}\Lp^{n_k}\psi_t\to\psi_t$, there exists $k$ with
$n_k>\ell$ such that $\lambda^{-n_k}\Lp^{n_k}\psi_t(x)>0$.
Hence there is a preimage $x'\in F^{-n_k}(x)$ such that $\psi_t(x)>0$.
Let $Z_1 \in \Z$ be the partition element containing $x'$.   By construction,
$Z_1$ does not iterate into a hole before reaching $\Delta_0$ (in $m=n_k-\ell$
iterates).   In particular, $F^mZ_1$ covers an element of $\Z_0^{im}$.
Since $F$ is mixing, there exists an $N_1 >0$ such that for each
$n \geq N_1$, $F^nZ_1 \supset \Delta_0$.

Since $\psi_t\in \B_M$, it follows that $\inf_{Z} \psi_t =: \kappa > 0$.
Note that for any $n \geq 0$, the definition of $\psi_t$ and
equation \eqref{eq:normalize} imply that
\begin{equation}
\label{eq:small diff}
\lambda^{-n} \Lp^n \psi_t = t \vf + (1-t) [\re(e^{i\omega n}f) + K\vf]/C'
= \psi_t + (1-t)\re((e^{i \omega n} -1)f)/C' .
\end{equation}
Choose $\ve < \frac{\kappa C'}{2|1-t| \|f\|_\infty}$ and define
$Q_\ve = \{ n \in \N : |e^{i\omega n} - 1| < \ve \}$.
Notice that $Q_\ve$ has bounded gaps, i.e., there exists a $K_1=K_1(\ve)$
such that
for any $n \geq K_1$, there is a $k \leq K_1$ such that $n-k \in Q_\ve$.

It is clear from \eqref{eq:small diff} that
\begin{equation}
\label{eq:kappa}
\lambda^{-n} \Lp^n \psi_t(x)
\; \geq \; \psi_t(x) - |1-t|\|f\|_\infty/C'\ve \; \geq \; \kappa/2
\end{equation}
for $n \in Q_\ve$ and $x \in Z_1$.

Fix $n \geq N_1 + K_1$ and choose $k$ such that $N_1 \leq k \leq N_1 + K_1$
and $n-k \in Q_\ve$.  Note that for any $\ell$, $\inf_{\Delta_\ell} g$ may be
0 if there are infinitely many $Z \subset \Delta_\ell$ with $R(Z)=1$.  However,
since we only require returns to finitely many $Z' \in \Z_0^{im}$ for finitely
many times, $N_1 \leq k \leq N_1 + K_1$, we may choose a set
$W \subset \bigcup_{\ell \leq N_1 + K_1} \Delta_\ell$
containing only finitely many $Z$ such that for each $x \in \Delta_0$
there is a point $y_1 \in Z_1$ such that $F^ky_1 = x$ and
$F^iy_1 \in W$ for $0 \leq i \leq k-1$.

Now using \eqref{eq:kappa} and the
fact that $n-k \in Q_\ve$, we estimate
\[
\begin{split}
\lambda^{-n}\Lp^n \psi_t(x)
& = \lambda^{-k} \Lp^k (\lambda^{k-n} \Lp^{n-k} \psi_t)(x)
  = \lambda^{-k} \sum_{F^ky=x} \lambda^{k-n} \Lp^{n-k} \psi_t(y) g_k(y) \\
& \geq \lambda^{-k} (\lambda^{k-n} \Lp^{n-k} \psi_t)(y_1) g_k(y_1)
\geq {\textstyle\frac12}\lambda^{-N_1} \kappa \inf_W g_{_{N_1+K_1}}
=: \kappa' >0.
\end{split}
\]
Thus $\inf_{\Delta_0} \lambda^{-n}\Lp^n \psi_t \geq \kappa'$ for all
$n \geq N_1 + K_1$.

Now on $\Delta_\ell$, for $n \geq \ell + N_1+K_1$,
$\lambda^{-n} \Lp^n \psi_t(x) = \lambda^{-\ell} \lambda^{\ell -n}
\Lp^{n-\ell} \psi_t(F^{-\ell}x) \geq \lambda^{-\ell} \kappa'$.
Therefore for large $n$,
$\inf_{\Delta_\ell} \lambda^{-n} \Lp^n\psi_t \geq \kappa'$
for all $\ell \leq n - N_1 - K_1$.
Since $\psi_t = \lim_k \lambda^{-n_k} \Lp^{n_k} \psi_t$, we have
$\inf_\Delta \psi_t \geq \kappa'$.
\end{proof}

\medskip
{\em 6. Mixing implies $N_0=1$.}
Suppose that $\Lp^{N_0}\varphi=\lambda^{N_0}\varphi$.
The proofs of Steps 2 and 4 go through with $\Lp$ replaced by $\Lp^{N_0}$, implying
that $\lambda^{N_0}$ is a simple eigenvalue for $\Lp^{N_0}$.  (The proofs are modified
in the obvious way.  For example, $\Delta_0$ is replaced by $\Delta_0\cup\dots\cup\Delta_{{N_0}-1}$ and mixing is used instead of transitivity.)
But $\Lp^{N_0}(\Lp\varphi)=\lambda^{N_0}\Lp\varphi$, so we deduce that
$\Lp\varphi=c\varphi$ for some $c\in\R$, with $c^{N_0}=\lambda^{N_0}$.
Positivity of $\Lp$ implies that $c>0$, so $c=\lambda$.
Hence $\Lp\varphi=\lambda\varphi$, that is $N_0=1$.

\medskip
{\em 7. Nonmixing case.}
First suppose that $F$ is transitive with period $p$.
Then $F^p$ has $p$ distinct
components in $\Delta$ and is mixing on each of them.
Applying {\em (i)} to $\Lp^p$
implies that $\lambda^p$ is an eigenvalue of algebraic and geometric
multiplicity $p$ and there are no further eigenvalues on or
outside the circle of radius $\lambda^p$.   The corresponding eigenvalues
for $\Lp$ lie at $p^{\mbox{\scriptsize th}}$ roots of $\lambda^p$,
and it follows easily from transitivity that all $p^{\mbox{\scriptsize th}}$
roots are realized by simple eigenvalues, proving {\em (ii)}.

Finally, since $\Lp$ is quasi-compact, there are only finitely many transitive
components of $\Delta$.  Restricting to a single
component, {\em (iii)} reduces to the transitive case {\em (ii)}.
\end{proof}

%%%%%%%%%%%%%%%%%%%%%%%%%%%%%%%%%%%%%%%%%%%%

\subsection{An Invariant Measure on $\Delta^\infty$}
\label{invariant measure}

\begin{proof}[Proof of Proposition~\ref{prop:inv measure}]
We assume that $F$ is mixing and as usual
denote by $\vf$ the unique eigenvector with eigenvalue $\lambda$.
We divide the proof into three parts.

\medskip
\noindent
\emph{(i) Existence of $\nu$.}
Let $f\in\B_0$.
Since $\vf|_{\Delta_\ell} \sim\lambda^{-\ell}$ where $\lambda>\beta$,
it follows from the definitions of $\B$ and $\B_0$ that $\vf f \in \B$.
By Corollary~\ref{cor-exp},
\begin{equation}
\label{eq:linear functional}
\ff(f) := \lim_{n \to \infty} \lambda^{-n} \vf^{-1} \Lp^n(\vf f)
= c(\vf f).
\end{equation}
Hence \eqref{eq:linear functional} defines a linear functional  $\ff:\B_0\to\R$.
We also have
$|\Lp^n(\vf f)| \leq |f|_\infty \Lp^n\vf=|f|_\infty \lambda^n\vf$, so that $|\ff(f)| \leq |f|_\infty$.

Since $\ff$ is a bounded linear functional on $\B_0$, there exists a measure $\nu$ such that
$\ff(f) = \int f d\nu$ for each $f$ in $\B_0$.
Since $\ff(1) =1$, $\nu$ is a probability measure.
Notice also that we can write
$\lambda^{-n}\Lp^n(\vf f)\to\vf\nu(f)$ where convergence takes place
in $\B$ and hence in $L^1(m)$.   Since $\int_\Delta \vf \,dm=1$, it follows
that
\[
\nu(f) = \lim_{n\to \infty} \lambda^{-n}
            \int_{\Delta} \Lp^n(\vf f)  \, dm
        = \lim_{n\to \infty} \lambda^{-n} \int_{\Delta^n} f \vf \, dm
            = \lim_{n\to \infty} \lambda^{-n} \int_{\Delta^n} f \, d\mu
\]
so that $\nu$ is supported on
$\Delta^\infty$.  Also, from \eqref{eq:linear functional} it follows that
$c(f) = \nu(\vf^{-1}f)$ for each $f \in \B_0$.

Note that $\Lp(\vf\,f\circ F)=f\Lp\vf=\lambda \vf f$ and so
\begin{eqnarray*}
\ff(f\circ F) & = & \lim_{n\to \infty} \lambda^{-n}\vf^{-1}\Lp^n(\vf\,f\circ F)
=\lim_{n\to \infty} \lambda^{-n}\vf^{-1}\Lp^{n-1}(\lambda \vf f) \\ &
= & \lim_{n\to \infty} \lambda^{-(n-1)}\vf^{-1}\Lp^{n-1}(\vf f)=\ff(f).
\end{eqnarray*}
Hence $\nu$ is an invariant measure for $F$ (and $\F$,
since $F = \F$ on $\Delta^\infty$).

\medskip
\noindent
\emph{(ii) $\nu$ is ergodic.}
Since $F$ is transitive on $\Z_0^{im}$,
given $Z'_1$, $Z'_2 \in \Z_0^{im}$, we may choose
$n \in \N$ such that $F^n(Z'_1) \supseteq Z'_2$.  Since $\Delta^\infty$ is
an $F$-invariant set, this implies that
$F^n(Z'_1 \cap \Delta^\infty) \supseteq Z'_2 \cap \Delta^\infty$.
So $F|_{\Delta^\infty}$ is transitive.

Let $Z^n_i \subset \Delta^\infty$ denote a cylinder set of length $n$ with
respect to the partition $\Z_0 \cap \Delta^\infty$.
Now suppose $A = \bigcup_{i,n} Z^n_i$ is a countable union of such cylinder
sets with $F^{-1}A = A$ and $\nu(A) > 0$.
Since $A$ is a countable
union, we must have $\nu(Z_i^n)>0$ for some $i$ and $n$.
This implies that
$F^n(Z^n_i) = Z \cap \Delta^\infty$ for some $Z \in \Z_0$, and
$F^{n+R(Z)}(Z^n_i) \supseteq Z' \cap \Delta^\infty$ for some
$Z' \in \Z_0^{im}$.  In particular, $Z' \cap \Delta^\infty \subset A$.
Since $F$ is transitive on $\Delta^\infty$,
$\cup_{k \geq 0}F^k(Z' \cap \Delta^\infty) = \Delta^\infty$.
Thus $A= \Delta^\infty$ so $\nu(A)=1$.

Since $\Z$ is a generating partition on $\Delta^\infty$, we conclude
that $\nu$ is ergodic.

\medskip
\noindent
\emph{(iii) Exponential decay of correlations.}
Let $f_1, f_2 \in \B_0$.  Recall that $\nu(f_1) = c(f_1 \vf)$.
By definition of $\nu$,
\[
\begin{split}
\int_{\Delta^\infty} f_1 f_2\circ F^n d\nu & - \nu(f_1) \nu(f_2)
  \; = \; \lim_{k \to \infty} \lambda^{-k}
          \int_{\Delta^k} f_1 f_2\circ F^n \vf \, dm
          - \int_{\Delta^\infty} \nu(f_1) f_2 \, d\nu \\
  & = \; \lim_{k \to \infty} \lambda^{-k}
          \int_{\Delta^{k-n}} \Lp^n(f_1 \vf) f_2 \, dm
          - \lim_{k \to \infty} \lambda^{n-k}
          \int_{\Delta^{k-n}} \nu(f_1) f_2 \vf \, dm \\
  & = \; \lim_{k \to \infty} \lambda^{n-k}
          \int_{\Delta^{k-n}} [\lambda^{-n} \Lp^n(f_1 \vf)
          - c(f_1 \vf) \vf ] \, f_2 \, dm  \\
  & = \; \lim_{k \to \infty}  \lambda^{-k} \sum_{\ell \geq 0}
          \int_{\Delta^{k} \cap \Delta_\ell} [\lambda^{-n} \Lp^n(f_1 \vf)
          - c(f_1 \vf) \vf ] \, f_2 \, dm  .
\end{split}
\]
Recall that $f_1\vf\in\B$.
For $F$ mixing, it follows from Corollary~\ref{cor-exp} that
 \begin{eqnarray*}
\Bigl|\int_{\Delta^{k} \cap \Delta_\ell} [\lambda^{-n} \Lp^n(f_1 \vf)
          - c(f_1 \vf) \vf ]f_2 \, dm\Bigr|
& \leq & | 1_{\Delta_\ell}(\lambda^{-n} \Lp^n(f_1 \vf) - c(f_1 \vf) \vf)|_\infty
|f_2|_\infty m(\Delta^{k} \cap \Delta_\ell) \\ &\le& \| \lambda^{-n} \Lp^n(f_1 \vf) - c(f_1 \vf) \vf\|\beta^{-\ell} |f_2|_\infty m(\Delta^{k} \cap\Delta_\ell) \\
& \leq & C \|f_1 \vf \| |f_2|_\infty \sigma^n \beta^{-\ell} m(\Delta^{k} \cap\Delta_\ell).
\end{eqnarray*}
Hence
\[
\begin{split}
\Bigl|\int_{\Delta^\infty} f_1 f_2\circ F^n d\nu & - \nu(f_1) \nu(f_2)\Bigr|
  \; \leq \;
\lim_{k \to \infty} \lambda^{-k} \sum_{\ell \geq 0}
       C \|f_1 \vf \| |f_2|_\infty \sigma^n \beta^{-\ell}
        m(\Delta^{k} \cap \Delta_\ell) \\
  & = \; C \|f_1 \vf \| |f_2|_\infty \sigma^n \lim_{k \to \infty} \lambda^{-k}
\int_{\Delta^k}f_\beta\,dm,
\end{split}
\]
where $f_\beta|_{\Delta_\ell} :=\beta^{-\ell}$.
In particular, $f_\beta\in\B$.
By Corollary~\ref{cor-exp}, $\lambda^{-k}\Lp^kf_\beta$ converges
to $c(f_\beta)\vf$ in $\B$, and hence in $L^1(m)$ so that
$\lim_{k\to\infty}\lambda^{-k} \int_{\Delta^k}f_\beta\,dm=c(f_\beta)$,
completing the proof.
\end{proof}

%%%%%%%%%%%%%%%%%%%%%%%%%%%%%%%%%%%%%%%%%%%%

\subsection{Escape rates from $\Delta$}
\label{escape}

Notice that the functional analytic approach adopted thus far only
tells us that $\lambda$ represents the slowest rate of escape from $\Delta$ for
elements of $\B$, but in general there are functions which escape at
faster rates.
The estimates on the functional $\ff$ in Section~\ref{invariant measure}
and the existence of the invariant
measure $\nu$ allow us to establish the uniformity of escape rates
for certain functions in $\B$.  Since the indicator functions of
elements of the partition $\Z$ are in this space,
we also obtain uniform escape
rates of mass from certain sets and in particular for the reference
measure $m$ on the tower.

\begin{proposition}
\label{prop:convergence criterion}
Let $F$ be mixing and satisfy properties (P1)-(P3) and (H1) of
Section~\ref{holes}.
For each $f \in \B_0$ with $f \geq 0$, we have $\nu(f) >0$ if and only if
\begin{equation}
\label{eq:limit}
\lim_{n \to \infty} \frac{\Lp^nf}{|\Lp^nf|_1} = \vf
\end{equation}
where as usual, the convergence is in the $\| \cdot \|$-norm.
In particular, the reference measure converges to the \accim
\end{proposition}
\begin{proof}
By Proposition~\ref{prop:convergence}, equation~\eqref{eq:limit} holds
if and only if $c(f)>0$.  Thus it suffices to prove $\nu(f) >0$ if and
only if $c(f) >0$.

Note that from the proof of Proposition~\ref{prop:inv measure},
$\nu(f) = c(\vf f) \geq \delta c(f)$ since
$\vf \geq \delta$.  So $c(f)>0$ implies $\nu(f) >0$ immediately.

Now fix $f \in \B_0$ and suppose $\nu(f) >0$.
Let $\Delta_\ell^n = \Delta_\ell \cap \Delta^n$
be the subset of $\Delta_\ell$ which has not escaped by time $n$.
Set $\Delta^n_{(K)} = \cup_{\ell = 0}^K \Delta_\ell^n$ and
$\Delta^n_+ = \Delta^n \backslash \Delta^n_{(K)}$.

For $\ve \in (0,1)$,
choose $K$ such that $\nu(\Delta^n_+) |f|_\infty < \ve \nu(f)$.
Then
\[
\begin{split}
\nu(f) & = \lim_{n \to \infty}
         \left( \lambda^{-n} \int_{\Delta^n_{(K)}} f\, d\mu
   + \lambda^{-n} \int_{\Delta^n_+} f\, d\mu \right)    \\
  & \leq \lim_{n \to \infty} \left( \lambda^{-n} \lambda^{-K} \delta^{-1}
             \int_{\Delta^n_{(K)}} f\, dm
   + \lambda^{-n} |f|_\infty \int_{\Delta^n_+}  d\mu \right)  \\
  & \leq \lambda^{-K} \delta^{-1} c(f)
   + |f|_\infty \nu(\Delta^n_+) \; \leq \; \lambda^{-K} \delta^{-1} c(f)
   + \ve \nu(f) .
\end{split}
\]
Since $\ve \in (0,1)$, we have $c(f)>0$.

Since $\nu(1)=1$, the normalized push forward of the reference measure
$m$ converges to $\mu$ as $n \to \infty$.
\end{proof}

\begin{cor}
\label{cor:uniform escape}
Let $A = \cup_{(\ell, j) \in J} Z_{\ell, j}$ be a union of partition elements such that
$\nu(A) > 0$.  Then there exists $C>0$ such that
\begin{equation}
\label{eq:uniform bound}
C^{-1} \lambda^n \leq m(\Delta^n\cap A) \leq C \lambda^n
\end{equation}
for each $n \in \mathbb{N}$ so that mass with respect
to $m$ escapes from $A$ at a uniform rate matching that of the
conditionally invariant measure.
\end{cor}

\begin{proof}
First note that for any $f \in \B$, we have
\[
|\Lp^n f|_1 \leq \| \Lp^n f\| \sum_{\ell \geq 0} \beta^{-\ell} m(\Delta_\ell)
\leq C \|\Lp^n\| \|f\| \leq C \lambda^n \|f\|,
\]
so that
the upper bound in \eqref{eq:uniform bound} is trivial.

Let $\chi_A$ be the indicator function for $A$ and notice that
$\chi_A \in \B_0$.  Integrating the limit in Corollary~\ref{cor-exp}, we get
\[
c(\chi_A) = \lim_{n\to\infty} \lambda^{-n} \int_\Delta \Lp^n\chi_A \, dm
= \lim_{n\to\infty} \lambda^{-n} \int_{\Delta^n} \chi_A \, dm .
\]
Since $c(\chi_A)>0$ by
Proposition~\ref{prop:convergence criterion} and
$m(\Delta^n \cap A)$ forms a decreasing sequence, there must exist
a $C>0$ such $\lambda^{-n} m(\Delta^n \cap A) \geq C^{-1}$ for all $n$.
\end{proof}

\begin{cor}
\label{cor:uniform cyl}
Let $Z = Z_{\ell,j}$ be a cylinder set and let $n > R(Z)$.  There
exists a constant $C>0$, independent of $Z$, such that if
$\Delta^n \cap Z \neq \emptyset$, then
\[
C^{-1} \lambda^{n-R} m(Z) \leq m(\Delta^n \cap Z) \leq C \lambda^{n-R} m(Z) .
\]
\end{cor}
\begin{proof}
By bounded distortion, we have
$m(\Delta^n \cap Z) |(F^R)'(y)| = m(\Delta^{n-R} \cap F^RZ)$ for some $y \in Z$.
Since $F^RZ = Z'$ for some $Z' \in \Z'$ and $\Z'$ is finite, by Corollary~\ref{cor:uniform escape}, we can find $C$ independent of $Z'$ such that
\[
C^{-1} \lambda^n \leq m(\Delta^{n-R} \cap Z') \leq C \lambda^n .
\]
We complete the proof by noting that $|(F^R)'(y)| \approx c_0/m(Z)$.
\end{proof}

\begin{cor}
\label{cor:convergence}
Let $f \in \B$, $f\geq 0$, such that $\nu(x \in \Delta: f(x)>0 ) > 0$.
Then
\[
\lim_{n\to \infty} \frac{\Lp^n f}{|\Lp^n f|_1} = \vf.
\]
\end{cor}
\begin{proof} Let $h = \min\{f, 1\}$ and note that $h \in \B_0$.
Also $\nu(h)>0$ by assumption on $f$ since $h$ and $f$ share the same
support.  Thus $c(h)>0$ by Proposition~\ref{prop:convergence criterion}.
Now
\[
c(f) = \lim_{n \to \infty} \lambda^{-n} \int_{\Delta^n} f \, dm
\geq \lim_{n \to \infty} \lambda^{-n} \int_{\Delta^n} h \, dm = c(h) > 0
\]
so the limit for $f$ holds by Proposition~\ref{prop:convergence}.
\end{proof}

%%%%%%%%%%%%%%%%%%%%%%%%%%%%%%%%%%%%%%%%%%%%%%%%%%%%%%%%%%%%
%%%%%%%%%%%%%%%%%%%%%%%%%%%%%%%%%%%%%%%%%%%%%%%%%%%%%%%%%%%%

\section{Applications}
\label{applications}

\subsection{General Approach}
\label{general approach}

We set up our notation as follows.
Let $\T$ be a piecewise $C^{1+\alpha}$ self-map of a metric space
$(\X, \td )$ with
open hole $\tH$.
Let $\tm$ be a probability measure on $\X$ and let
$\tg = \frac{d\tm}{d(\tm \circ \T)}$.
Suppose that a tower $(\F,\HDelta)$ with hole $H$ and the properties of
Section~\ref{holes}
can be constructed over a reference set $\Lambda$.
This implies that there exists a countable partition $\Z_0$
of $\Lambda$, a coarser partition $\Z_0^{im}$, also of $\Lambda$,
and a return time function $R$ which is constant on elements
of $\Z_0$ and for which $\T^R(Z) \in \Z_0^{im}$
or $\T^R(Z) \subset \tH$ for each $Z \in \Z_0$. The set $\Lambda$
is identified with $\Delta_0$ and each level
$\Delta_\ell$ is associated with $\cup_{R(Z) > \ell} T^\ell(Z)$.  This defines a natural
projection
$\pi:\HDelta \to \hat{X}$ so that $\pi \circ \F^n = \T^n \circ \pi$ for each $n$.
In general, we may choose $\Lambda$ so that $\Lambda \cap H = \emptyset$.

Following our previous notation, we define $X = \hat{X} \backslash \tH$
and $X^n = \cap_{i=0}^n \T^{-i}X$.
The restricted maps are then $F^n = \F^n|_{\Delta_n}$ on the tower and
$T^n = \T^n|_{X^n}$ on the underlying space.

We use the reference measure $\tm$ on $\hat{X}$ to define a reference
measure $m$ on $\Delta$ by letting
$m|_{\Delta_0} = \tm|_{\Lambda}$ and then simply defining $m$ on subsequent
levels by
$m(A) = m(\F^{-\ell}A)$ for measurable $A \subset \HDelta_\ell$.
As before, we let $g = \frac{d m}{d(m \circ \F)}$.

Given a measure $\mu$ on $\HDelta$, we define its projection
$\tmu$ onto $\hat{X}$, by $\tmu = \pi_*\mu$.
In terms of densities, this implies
that if $d\mu = f dm$, then for almost every $u \in \hat{X}$,
the density $\tf$ of $\tilde{\mu}$ is given by
\[
\ppi f(u) = \sum_{x \in \pi^{-1}u} f(x)/J\pi(x)
\]
where $J\pi = \frac{d(\tm \circ \pi)}{dm}$.  Note that
$|\ppi f|_{L^1(X,\tm)} = |f|_{L^1(\Delta,m)}$.
Since Radon-Nikodym derivatives multiply, we have
\begin{equation}
\label{eq:nik}
\tg_n(\pi y)/J\pi(y) = g_n(y)/J\pi(\F^ny)
\end{equation}
for almost every $y \in \HDelta$ and each $n \geq 0$.
This in turn implies that
\begin{equation}
\label{eq:transfer commute}
\ppi(\Lp^n_F f) = \Lp_T^n (\ppi f)
\end{equation}
for $f \in L^1(\Delta)$.
\iffalse
To see this, take $u \in X$ and write on
the one hand,
\[
\ppi(\Lp^n_F f)(u)
                   = \sum_{x \in \pi^{-1}u} (J\pi(x))^{-1}
                                \sum_{y \in F^{-n}x} f(y)g_n(y)
                  = \sum_{y \in F^{-n} \circ \pi^{-1} u} f(y)g_n(y)/J\pi(F^ny),
\]
while on the other hand,
\[
\Lp^n_T (\ppi f)(u)
                      = \sum_{v \in T^{-n}u} \tg_n(v)
                         \sum_{y \in \pi^{-1}v} f(y)/J\pi(y)
                  = \sum_{y \in \pi^{-1} \circ T^{-n} u} f(y)\tg_n(\pi y)/J\pi(y).
\]
The two expressions are equal due to \eqref{eq:nik}.
\fi
The importance of these relations lies in the fact that if
$\vf$ satisfies $\Lp_F \vf = \lambda \vf$ and  $\tf = \ppi f$, then
\begin{equation}
\label{eq:conv in X}
\frac{\Lp^n_Ff}{|\Lp^n_Ff|_1} \to \vf \qquad \mbox{in $L^1(m)$ implies} \qquad
\frac{\Lp_T^n \tf}{|\Lp^n_T \tf|_1} \to \ppi \vf =: \tilde{\vf}
\qquad \mbox{in $L^1(\tm)$}
\end{equation}
and $\tilde{\vf}$ satisfies
$\Lp_T \tilde{\vf} = \ppi (\Lp_F \vf) = \lambda \tilde{\vf}$ so that
$\tilde{\vf}$ defines a conditionally invariant measure for
$T$ with the same eigenvalue as $\vf$.

However, the space $\ppi \B$ is not well understood
and functions in $\ppi \B$ are a priori no better than $L^1$.  It is not even clear that
the constant function corresponding to the original reference measure $\tm$ is in $\ppi \B$.
Getting a handle on a nice class of functions in $\ppi \B$ is necessary for showing in particular
applications that, for example, Lebesgue measure converges to the \accim\ according to the results
of the previous section.

In what follows, we identify two properties, (A1) and (A2), that
guarantee $C^{\bar \alpha}(X) \subset \ppi \B$ where $\bar \alpha$ depends
on the smoothness and average expansion of $T$.  (A1) is standard in
constructions of Young towers and (A2) can be achieved with
no added restrictions on the map or types of holes allowed.
In Sections~\ref{expanding} and \ref{logistic}, we prove that
the towers we construct have these properties.

Let $R_n(x) = R_{n-1}(T^{R(x)}(x))$ be the $n^{\mbox{\scriptsize th}}$ good return
of $x$ to $\Lambda$, for $n \geq 1$.

\medskip
\noindent
\parbox{.07 \textwidth}{\bf(A1)}
\parbox[t]{.9 \textwidth}{ There exist constants $\tau > 1$ and $C_2$, $C_3 > 0$ such that
\begin{itemize}
  \item[(a)] for any $x \in \Lambda$, $n \geq 1$ and $k<R_n(x)$,
  $|DT^{R_n(x)-k}(T^kx)| > C_2 \tau^{R_n(x)-k}$.
  \item[(b)] Let $x,y \in Z_{0,j}$ and $R = R(Z_{0,j})$.  Then
  $\left| \frac{\tg_\ell(\pi x)}{\tg_\ell(\pi y)} \right| \leq C_3$ for $\ell \leq R$.    If
  $T^R(Z_{0,j}) \subseteq \Lambda$, then $\left| \frac{\tg_R(\pi x)}{\tg_R(\pi y)} -1 \right|
        \leq C_3 d(T^R(\pi x),T^R(\pi y))^{ \alpha}$.
\end{itemize} }

Property (A1)(a) says that although $T$ may not be expanding everywhere in its phase space, we only count
returns to $\Lambda$ during which average expansion has occurred.  Property (A1)(b) is simply bounded
distortion.
In fact, (A1) implies the distortion bound~\eqref{eq:distortion} as well as (P2)
in the towers we use.

%%%%%%%%%%%%%%%%%%%%%%%%%%%%%%%%%%%%%%%%%%%%%%%%%%%%%%%%%%%%%%%

\subsubsection{Lifting H\"{o}lder functions on $X$}

Recall that $\td$ is the metric on $X$ and $d$ is the symbolic metric on $\Delta$ defined in Section~\ref{holes}.
Under assumption (A1)(a),
these two metrics are compatible in the following sense.

\begin{lemma}
\label{lem:lift}
For any $\bar \alpha \geq - \log \beta/\log \tau$,
let $\tf \in C^{\bar \alpha}(X)$
and define $f$ on $\Delta$ by $f(x) = \tf(\pi x)$ for each $x \in \Delta$.
Then $f \in \B_0$ and $\|f\|_0 \leq C_2^{-1} |\tf|_{C^{\bar \alpha}}$.
\end{lemma}

\begin{proof}
First we show that $\mbox{Lip}(f) = \sup_{\ell, j} \mbox{Lip}(f_{\ell,j}) < \infty$.
Let $x,y \in \zlj$ and let $\tx = \pi x$ and $\ty = \pi y$. Then
\begin{equation}
\label{eq:holder}
\frac{|f(x)-f(y)|}{d(x,y)} = \frac{|\tf(\tx) - \tf(\ty)|}{\td(\tx,\ty)^{\bar \alpha}} \cdot
    \frac{\td(\tx,\ty)^{\bar \alpha}}{d(x,y)}
    \leq C_{\bar \alpha, \tf} \frac{\td(\tx,\ty)^{\bar \alpha}}{d(x,y)}.
\end{equation}
Note that $d(x,y) = \beta^{s(x,y)}$ and that $s(x,y)$ is a return time for $\tx$ and $\ty$
so that $|DT^{s(x,y)}| \geq C_2 \tau^{s(x,y)}$ on $\zlj$ by Property (A1)(a).  Thus
\[
\td(\tx,\ty) = \frac{\td(\tx,\ty)}{\td(T^{s(x,y)}(\tx),T^{s(x,y)}(\ty))} \td(T^{s(x,y)}(\tx),T^{s(x,y)}(\ty))
        \leq C_2^{-1} \tau^{-s(x,y)} \mbox{diam}(\Lambda).
\]
This, together with \eqref{eq:holder}, implies that Lip$(f) < \infty$ since
$\beta \geq \tau^{-\bar \alpha}$.
Also $|f|_\infty = |\tf|_\infty < \infty$, so $f \in \B_0$.
\end{proof}

The problem is that in general $\ppi(\tf \circ \pi) \neq \tf$, so Lemma~\ref{lem:lift}
does not imply that $C^\alpha(X) \subset \ppi \B$ immediately.

%%%%%%%%%%%%%%%%%%%%%%%%%%%%%%%%%%%%%%%%%%%%%%%%%%%%%%%%%%%%%%%

\subsubsection{A lift compatible with $\ppi$}
\label{compatible lift}

Given $\tf \in C^\alpha(X)$, we want to construct $f \in \B$ so that $\ppi f = \tf$.  To do this,
it is sufficient to have the following property on the tower constructed above the reference set $\Lambda$.

\medskip
\noindent
\parbox{.1 \textwidth}{\bf(A2)}
\parbox[t]{.9 \textwidth}{ There exists an index set $J \subset \mathbb{N} \times \mathbb{N}$ such that
\begin{itemize}
  \item[(a)]  $\tm(X \backslash \cup_{(\ell, j) \in J} \pi(\zlj)) = 0$;
  \item[(b)]  $\pi(Z_{\ell_1, j_1}) \cap \pi(Z_{\ell_2, j_2}) = \emptyset$ for all but finitely many
    $(\ell_1, j_1)$, $(\ell_2, j_2) \in J$;
  \item[(c)]  Define $J\pi_{\ell, j} := J\pi|_{Z_{\ell,j}}$.
  Then $\sup_{(\ell,j) \in J} |J\pi_{\ell,j}|_\infty +  \mbox{Lip}(J\pi_{\ell,j}) =  D < \infty$.
\end{itemize}  }

\begin{proposition}
\label{prop:continuous}
Let $T$ be a piecewise $C^{1+\alpha}$ self-map of a metric space $(X, \td)$ with hole $\tH$.  Suppose we can
construct a Young tower over a reference set $\Lambda$ for which $T$ satisfies properties (A1) and (A2).
Then $C^{\bar \alpha}(X) \subset \ppi \B_0$
for every $-\log \beta / \log \tau \leq \bar \alpha \leq \alpha$.
\end{proposition}

\begin{proof}
Let $\tf \in C^{\bar \alpha}(X)$ be given.

If $\pi(Z_{\ell,j}) \cap \pi(Z_{\ell',j'}) = \emptyset$ for all other $(\ell',j') \in J$,
then we can
choose a single preimage for each $u \in \pi(Z_{\ell,j})$ on which
to define $f$.
In fact, inverting the projection operator $\ppi$, we see that defining
$f(x) = \tf(\pi x)J\pi(x)$ for each
$x \in \zlj$ yields the correct value for $\tf(\pi x)$.

Now consider the case in which $\pi(Z_{\ell_1, j_1}) \cap \pi(Z_{\ell_2, j_2}) \neq \emptyset$.
We may choose
a partition of unity $\{ \rho_1, \rho_2 \}$ for $E = \pi(Z_{\ell_1, j_1} \cup Z_{\ell_2, j_2})$ such that $\rho_i \in C^\alpha(E)$.  Then we define $f$ by
\[
f_{\ell_i, j_i} (x_i) = \tf(\pi x_i) J\pi(x_i) \rho_i(\pi x_i)
\]
for $x_i \in Z_{\ell_i, j_i}$ and $i = 1,2$.  Then for $u \in E$, we set $f = 0$ on preimages of $u$ which
are not in $Z_{\ell_1, j_1} \cup Z_{\ell_2, j_2}$.  It is clear that $\ppi f(u) = \tf(u)$ for $u \in E$.

This construction can be generalized to accommodate finitely many overlaps in the projections $\pi(\zlj)$
while maintaining a uniform bound on the $C^\alpha$-norm of the $\rho_i$.

Let $\Z_J = \cup_{(\ell,j) \in J} \zlj$. Lemma~\ref{lem:lift} tells us that $\tf \circ
\pi \in \B_0$ (where $\B_0$ is defined in \eqref{eq:B0}) and (A2)(c) implies that
$J\pi|_{\Z_J} \in \B_0$. Since $f \equiv 0$ outside of $\Z_J$, it follows immediately
that $f \in \B_0$.
\end{proof}

\iffalse
so if $J\pi \in \B_0$,
it would follow immediately that $f$ belongs to $\B_0$.
Property (A2)(c) guarantees that
$|J\pi_{\ell,j}|_\infty \leq D$ so it remains to estimate the Lipschitz constant of $J\pi$.

For $x, y \in \zlj$, let $x_0 = F^{-\ell}(x)$ and $y_0 = F^{-\ell}(y)$.
Then \eqref{eq:nik} implies
\[
J\pi(x) = \frac{g_\ell(x_0)}{\tg_\ell (\pi x_0)} J\pi(x_0).
\]
Since $x_0 \in \Delta_0$, $J\pi(x_0)=g_\ell(x_0)=1$ so that
$J\pi(x) = 1/\tg_\ell (\pi x_0)$.  Using (A2)(c), we
estimate
\[
|J\pi(x) - J\pi(y)| \leq D  \left| \frac{J\pi(y)}{J\pi(x)} -1 \right|
= D \left| \frac{\tg_\ell(\pi x_0)}{\tg_\ell(\pi y_0)} -1 \right|
\leq C \beta^{-\ell}  \frac{\td(T^\ell(\pi x_0), T^\ell(\pi y_0))^\alpha }{d(x,y)} d(x,y)
\]
where in the last line we have used (A1)(b).  Since $T^\ell(\pi x_0) = \pi x$ and similarly for $y$, the fraction
on the right hand side of the inequality is bounded by a constant using the same argument following \eqref{eq:holder}
in the proof of Lemma~\ref{lem:lift}.

These estimates imply that $\sup_{(\ell, j) \in J} \mbox{Lip}(J\pi_{\ell, j}) < \infty$
so that $f \in \B_0$.
\end{proof}

\fi

%%%%%%%%%%%%%%%%%%%%%%%%%%%%%%%%%%%%%%%%%%%%%%%%%%%%%%%%%%%%%%%

\subsection{Piecewise Expanding Maps of the Interval}
\label{expanding}

\begin{proof}[Proof of Theorem~\ref{thm:exp convergence}]
Theorem A guarantees that $T$ admits a tower $(F,\Delta)$ satisfying
properties (P1)-(P3) and (H1).  Property (A1) is automatic for expanding maps.

It remains to verify that Property (A2) is satisfied.
This follows from the tower construction contained in \cite{demers exp}.
For this class of maps, we may choose the reference set
$\Lambda$ to be an interval of monotonicity of $T$ and the finite
partition of
images $\Z_0^{im}$ will consist of the single element $\Lambda$, i.e., we have
a tower with full returns to the base.
In the inductive construction of the
partition $\Z_0$ on $\Lambda$, at each step, new pieces are created only by
intersections with discontinuities,
intersections with the hole, and returns to the base.  In this way, only
finitely many distinct
pieces are generated by each iterate and therefore we have only
finitely many overlaps when we
project each level.  Since $I$ is covered in finitely
many iterates of $\Lambda$ by assumption (T1), it is also
covered by the projection of finitely
many levels of $\Delta$, say the first $N$.
Thus if we take our index set $J$ to be all indices corresponding to elements
in the first $N$ levels of the tower, it is immediate that (A2)(a) and (A2)(b) are
satisfied.

To see that (A2)(c) is satisfied, let $x \in \Delta_0$ and notice that
by \eqref{eq:nik},  $J\pi(F^\ell x) = J\pi(x) g_\ell(x) / \tg_\ell(\pi x)$.
If $\ell < R(x)$,
then $J\pi(x) = g_\ell(x) =1$ so that
\begin{equation}
\label{eq:pi}
J\pi(F^\ell x) = 1 / \tg_\ell(\pi x) = |(T^\ell)'(\pi x)|.
\end{equation}
Since $T$ is $C^{1+\alpha}$, so is $T^\ell$ for each $\ell$.
Since we are only concerned with $\ell \leq N$ and $\bar \alpha \leq \alpha$,
by Lemma~\ref{lem:lift},
$J\pi|_{\Z_J} \in \B_0$ so (A2)(c) is satisfied.
By Proposition~\ref{prop:continuous}, we have $C^\alpha(X) \subset \ppi B$.

Property (T1) also implies that we can construct $(F,\Delta)$ to be mixing,
since if $T^n(Z') \supseteq I$, then $T^{n+1}(Z') \supset I$ so we can avoid
periodicity in the return time $R$ by simply delaying a return by $1$ step.
Applying Proposition~\ref{prop:spectral picture}, we see that $\Lp_F$
admits a unique probability
density $\vf$ for the eigenvalue $\lambda$ of maximum modulus.
Defining $\tp = \ppi \vf$, we have
$\frac{\Lp_T^n \tf}{|\Lp_T^n \tf|} \to \tp$ at an exponential rate for every $\tf \in \ppi \B$ for which $c(\tf) > 0$ by Proposition~\ref{prop:convergence}.

\medskip
\noindent
{\em Convergence property.}
Let $\tf \in \G$.
Since $(F,\Delta)$ satisfies (A1) and (A2) and $\tf \in C^\alpha(I)$,
by Proposition~4.2 we can find
$f \in \B_0$, supported entirely in elements corresponding to the index set $J$,
such that $\ppi f = \tf$.
By Corollary~\ref{cor:convergence}, it suffices to show that  $\nu(f) > 0$, for
then the convergence of $f$ to $\vf$ will imply the convergence of
$\tf$ to $\tp := \ppi \vf$.

Since $\tf \in \G$, we have
$\tf > 0$ on $I^\infty \cap \Lambda$ which implies
$f>0$ on $\Delta^\infty \cap \Delta_0$.  Since $\nu$ is an
invariant measure on $\Delta$, it must be that $\nu(\Delta_0)>0$
and so $\nu(f) >0$ as required.

\medskip
\noindent
{\em  Unified escape rate.}
Finally we prove that all functions in $\G$ have the same escape rate given
by $-\log \lambda$.
First note that  given $\tf \in \G$ and $f \in \B$ such that $\ppi f = \tf$, we have
\[
\lim_{n\to \infty} \lambda^{-n} \Lp_T^n \tf
= \lim_{n\to \infty} \lambda^{-n} | \Lp_T^n \tf |_1 \,
\frac{\Lp_T^n \tf}{ |\Lp_T^n \tf|_1}
= \lim_{n\to \infty} \lambda^{-n} | \Lp_F^n f |_1 \,
\frac{\Lp_T^n \tf}{ |\Lp_T^n \tf|_1}
= c(f) \tp
\]
by Corollary~\ref{cor-exp} and the proof of convergence above.
Since $\nu(f)>0$, we also have $c(f)>0$.

Thus if we let $\teta = \tf \tm$, we have
\[
\lim_{n\to \infty} \frac{1}{n} \log \teta(I^n) =
\lim_{n\to \infty} \frac{1}{n} \log |\Lp_T^n \tf|_1 = \log \lambda  .
\]

\end{proof}

%%%%%%%%%%%%%%%%%%%%%%%%%%%%%%%%%%%%%%%

\subsection{Multimodal Collet-Eckmann Maps with Singularities}
\label{logistic}

\begin{proof}[Proof of Theorem~\ref{thm:CE convergence}]
The construction in \cite{DHL} fixes $\delta$ and finds an interval
$I^*$ with $c^* \in I^* \subset (c^*-\delta, c^*+\delta)$
as base for the induced map.
We choose $I^*$ such that
$\text{orb}(\partial I^*)$ is disjoint from the interior of $I^*$.
This is always possible by choosing $\partial I^*$ to be pre-periodic.
Now by using $I^*$ as the base $\hDelta_0$ of the Young tower
$\hDelta$ (i.e., without hole),
and recalling that $\Z$ is the natural partition of the tower
we have the following:
\begin{equation}
\label{eq:disjoint}
\mbox{For any $Z, Z' \in \Z$, the symmetric difference
$\pi Z\ \triangle\ \pi Z' = \emptyset$.}
\end{equation}

To show why this is true, write $Z = Z_{\ell,j}$ and $Z' = Z_{\ell',j'}$,
so $\pi(Z) = T^{\ell}(\pi Z_{0,j})$ and $\pi(Z') = T^{\ell'}(\pi Z_{0,j'})$.
Assume without loss of generality  that
$k := R(Z) - \ell \geq R(Z')-\ell' =: k'$.
If \eqref{eq:disjoint} fails, then there are
$x \in \partial \pi Z  \cap \pi Z'$ and
$x' \in \pi Z \cap \partial \pi Z'$. But then
$\T^k(x')$ is an interior point of $I^*$,
but at the same time
$\T^k(x') = \T^{k-k'}(\T^{k'}(x')) \in \T^{k-k'}(\partial I^*)$.
This contradicts the choice of $I^*$.
We record property \eqref{eq:disjoint} for later use in checking
condition (A2)(b).
\medskip

Next we adapt the construction of the inducing for the system without hole from
\cite{DHL}. By (B2)(a) the artificial critical points $b_j \in \tH_j$ satisfy $\T^k(b_j)
\neq c^*$ for all $k \geq 0$. Therefore (C3) still holds with the artificial critical
points. We set the binding period of $x \in B_\delta(b_j)$ (see \cite[Section 2.2]{DHL})
to $p(x) = \tau(j)-1$ for $\tau(j)$ as in \eqref{eq:tau}. Recall that the hole $\tH$ has
$L$ components. When the image $\T^n(\omega) = \omega_n$ of a partition element $\omega$
visits $B_\delta(b_j)$ (see \cite[page 432]{DHL}), we subdivide $\omega$ only if
$\omega_n$ intersects $\partial \tH_j$. If $\omega_n$ has not escaped to large scale, so
$|\omega_n| < \delta$, this results in at most $3$ subintervals $\omega' \subset \omega$
such that $\T^n(\omega')$ is either contained in $\tH_j$ or disjoint from $\tH$. By
\eqref{eq:tau} and our choice of binding period $p|_{B_{\delta}(b_j)}$, Lemma 2 of
\cite{DHL} is automatically satisfied for $\theta^* := \theta = \hat \theta =
\lambda^*$.

\begin{remark}
\label{rem:essential return}
In \cite{DHL} close visits to $\Crit_c \cup \Crit_s$ that result
in a cut are called {\em essential returns}, whereas those
that do not result in a cut are called {\em inessential returns}.
Let us call cuts caused by $\partial \tH$ {\em hole returns}.
The cutting of $\omega'$ at preimages of $\partial \tH_j$
is crucial for our tower to be compatible with the hole.
Note also that by the slow recurrence condition (B1), a cutting of $\omega'$
cannot occur within a binding
period after a previous visit to a point in $\Crit_c$.
\end{remark}

With this adaptation, the tower construction of \cite{DHL}
yields a tower $\hDelta$ and a
return time function $\hat R$, constant on elements of the partition $\Z$.
According to
\cite[Theorem 1]{DHL}, the tower $(\F, \HDelta, \hat R)$ satisfies
(P1)-(P3) and (A1) of the present paper.

Notice that at this point there is no escape.  We have simply introduced new
cuts at the boundaries of the hole during the construction of the return time
function and partition of the interval $I^*$ so that
the induced tower respects the boundary of the hole in the following sense:
For each
$Z \in \Z$, either $\pi Z \subset \tH$ or $\pi Z \cap \tH = \emptyset$.

A crucial feature of this construction is that the
exponential rate $\theta$ of the tail behavior is independent
of the size of $\tH$ when $\tH$ is small.  To see this, recall the notation introduced
in Section~\ref{CE results} regarding small holes.  We first fix the set of points
$b_1, \ldots b_L$,  which we regard as infinitesimal holes
satisfying (B1) and (B2).  Then for each $h>0$, the
family of holes $\Ho(h)$ consists of those holes $\tH$ satisfying:
(1) $b_j \in \tH_j$ and
$\tm(\tH_j) \leq h$ for each $1 \leq j \leq L$; and (2) $\tH$ satisfies (B1).

For the infinitesimal hole $\tH^{(0)}$ with
components $\tH_j^{(0)} = b_j$, $j = 1, \ldots, L$,
we fix $\delta >0$
and a reference interval
$I^* \subset B_\delta(c^*)$ and construct a tower $\Delta^{(0)}$
incorporating the additional cuts at $\partial \tH^{(0)}$ as described
above.

An immediate concern is that the presence of additional cuts when we introduce
holes of positive size interferes with returns to the extent that all full returns to
$I^*$ are blocked.  The following lemmas guarantee that this is not the case
and in fact several properties such as mixing and the rate of returns persist
for small holes.

\begin{lemma}
\label{lemma:good returns}
For sufficiently small $h$, each $\tH \in \Ho(h)$ induces
a tower $\hDelta^{(\tH)}$ and return time function $\hat R^{(\tH)}$ over
$I^*$ using the construction described above.
Moreover, $(\F^{(\tH)}, \hDelta^{(\tH)})$ is mixing if
$(\F^{(0)}, \hDelta^{(0)})$ is mixing.
\end{lemma}

\begin{proof}
Notice that the thickening of the hole at the points $b_j$
cannot affect returns which happen before a fixed time $n_h$ depending
only on $h$. For suppose $\omega \subset I^*$ satisfies
$\T^n \omega = I^*$ where $n = \hat R^{(0)}(\omega) \leq n_h$
is the return time
corresponding to $\tH^{(0)}$.  Then in fact $\omega$ is
in the middle third of a larger interval $\omega'$ such that
$\T^n \omega' \supset I^*$.  Cuts made by $\partial \tH_j$ must necessarily be at the
endpoints of $\omega'$ so for sufficiently small $h$,
the return of $\omega$ will still take place at time $n$.
By (B2)(a), we can force $n_h \to \infty$
as $h \to 0$, guaranteeing the persistence of returns up to any finite time for sufficiently small $h$.

To show $(\F^{(\tH)}, \hDelta^{(\tH)})$ is mixing, we need only show that
g.c.d.$(\hat R^{(\tH)})=1$ since $\hDelta^{(\tH)}$ has a single base.
Since $(\F^{(0)}, \hDelta^{(0)})$ is mixing, there exists $N$ such that
g.c.d.$\{ \hat R^{(0)}: \hat R^{(0)} \leq N\} =1$.
Now take $h$ small enough that $n_h \geq N$.
Then g.c.d.$\{ \hat R^{(\tH)} \} = 1$ as well.
\end{proof}

Our next lemma shows that the rate of return is uniform for small $h$.

\begin{lemma}
\label{lemma:uniform decay}
There exist
$\theta < 1$ and $C>0$ such that $\tm(\hat R^{(\tH)} > n) \leq C \theta^n$
for all $\tH \in \Ho(h)$ with $h$ sufficiently small.
\end{lemma}

\begin{proof}
Let $\theta_0$ be the exponential
rate of the tail behavior corresponding to $\tH^{(0)}$.
We will show that by choosing $\delta$ and $h$ sufficiently small, we can make
$\theta = \theta(\tH)$ arbitrarily close to $\theta_0$ for all
$\tH \in \Ho(h)$.
We do this by showing that the rates of decay given by a series of
lemmas in \cite{DHL} vary little for small $h$.

{\bf Lemma 1 of \cite{DHL}:}  Choose $h$ small enough that $n_h$
from the proof of Lemma~\ref{lemma:good returns} satisfies $n_h \geq t^*$ in
Lemma 1.  Then Lemma 1 holds with the same rate since returns in the
middle of large pieces are not affected by the hole before time $n_h$.

{\bf Lemma 6 of \cite{DHL}:}  The notation ${\mathcal E}_{n,S}(\omega)$
stands for the set of subintervals
within an interval $\omega$ of size $\delta/3 < |\omega| < \delta$
that have not grown to size $\delta$ by time $n$, and have essential return
depths summing to $S$
within these $n$ iterates.

This lemma estimates the size of any interval
$\omega' \in \E_{n,S}$.
Let us denote the number of hole returns used in the history
of $\omega'$ by $S_{\mbox{\tiny hole}}$.  Define
$\E_{n,S,S_{\mbox{\tiny hole}}}$ to be the set of subintervals
$\omega' \in \E_{n,S}$ such that $\omega'$ has $S_{\mbox{\tiny hole}}$
hole returns in its history up to time $n$.

Every hole return, i.e., a cut at $\partial \tH_j$, is followed by a
binding period of length $\tau(j)$ in which derivatives grow by an extra
factor of $6$ by \eqref{eq:tau}.
Since Lemma 6 is concerned only with derivatives, and not with the actual
cutting, the conclusion of Lemma 6 becomes:
For every $n \geq 1$, $S \geq 1$ and $S_{\mbox{\tiny hole}} \geq 1$ and
$\omega' \in {\mathcal E}_{n,S,S_{\mbox{\tiny hole}}}$ we have
\begin{equation}\label{eq:Lemma6}
|\omega'| \leq \kappa^{-1} e^{-\theta^* S} 6^{-S_{\mbox{\tiny hole}}}. \nonumber
\end{equation}
where $\theta^*$ replaces the $\theta$ used in \cite[Lemma 2]{DHL}.

{\bf Lemma 7 of \cite{DHL}:}
This lemma relies on combinatorial estimates to
obtain an upper bound on the number of pieces which can grow to
size $\delta$ at specific times.
By specifying the number of hole returns by $S_{\mbox{\tiny hole}}$
and using the fact that intervals are cut into at most $3$ pieces
during a hole return,
we can adapt the conclusion of Lemma 7 to
\[
\# {\mathcal E}_{n,S,S_{\mbox{\tiny hole}}}(\omega) \leq e^{\tilde \eta S} 3^{S_{\mbox{\tiny hole}}}.
\]

Combining Lemmas 6 and 7 in this form gives
\begin{eqnarray}\label{eq:21}
|{\mathcal E}_{n,S}(\omega)| &=& \sum_{S_{\mbox{\tiny hole}} \geq 0}|{\mathcal E}_{n,S,S_{\mbox{\tiny hole}}}(\omega)| \nonumber \\
&\le& \sum_{S_{\mbox{\tiny hole}} \geq 0}  \kappa^{-1} e^{-\theta^* S}
6^{-S_{\mbox{\tiny hole}}} e^{\tilde \eta S} 3^{S_{\mbox{\tiny hole}}} = \kappa^{-1} e^{-(\theta^* - \tilde \eta)S}
\end{eqnarray}
which is precisely formula (21) in \cite{DHL}.

The {\em free time} of an interval $\omega'$ are all the iterates
not spent in a binding period.
We suppose that $\omega'$ escapes to `large scale'
at time $n$ (i.e., $|\T^n(\omega')| \geq \delta$) and consider its history until time
$n$.
If $\omega'$ is cut very short at a hole return, say
$\partial \T^m(\omega') \cap \partial \tH_j \neq \emptyset$,
then we first have a binding period of length $p(x) = \tau(j)-1$, and the
free period after that lasts until either: (i) $\omega'$ reaches large scale,
(ii) $\omega'$ has the next artificial cut near $b_j \in \Crit_{\mbox{\tiny hole}}$, (iii) $\omega'$ has an inessential return near
$c \in \Crit_c \cup \Crit_s$.
or (iv) $\omega'$ has the next essential return near
$c \in \Crit_c \cup \Crit_s$.  In case (iv), $\T^k(\omega')$ covers at least
three intervals in the exponential partition of $B_\delta(c)$
as in \cite[page 433]{DHL}.

Let us call the time from iterate $m+\tau(j)-1$ to the next occurrence
of (i), (ii) or (iv) the {\em extended free period} of $\omega$
after iterate $m+\tau(j)-1$ .
So this includes binding and free periods after inessential
returns to $B_\delta(c)$ for $c \in \Crit_c \cup \Crit_s$.
These are the returns where
$\T^k(\omega')$ is too short to result in a cut.
Condition (B1) implies that when such an inessential return occurs,
the next cut or inessential return will not occur until after
the binding period associated to $\text{dist}(\T^k(\omega'), c)$.
This binding period will restore the small derivative incurred at time $k$
due to \cite[Lemma 2]{DHL}.
Hence there is $\lambda_{\mbox{\tiny hole}}$, depending only on $\lambda^*$ and $\Lambda^*$
from conditions (C1) and (C2), such that
\begin{equation}
\label{eq:growth in n_hole}
|D\T^\ell(x)| \geq e^{\lambda_{\mbox{\tiny hole}} \ell}
\end{equation}
for each $x \in \T^{m+\tau(j)-1}(\omega')$ and
$\ell$ is the length of this extended free period.
We let $n_{\mbox{\tiny hole}}$ denote the sum of extended
free periods directly following the binding periods due to hole returns in the
history of $\omega'$ up until time $n$.
With this notation, we make adaptations to the remaining lemmas.

{\bf Lemma 8 of \cite{DHL}:} This lemma can be changed to: there exists $n_\delta$
such that for every $\omega' \in {\mathcal E}_{n,0}$
with $S_{\mbox{\tiny hole}} = 0$, $\omega' = \omega$ and $n \leq n_\delta$.
The reason is that intervals of definite size cannot remain small forever
if they are not cut during an essential return or hole return,
and in fact the $n_\delta$ can be taken equal to the $n(\delta)$
used in condition (B2).

{\bf Lemma 9 of \cite{DHL}:} This lemma can be restated as:
for all $n \geq 1$ and $S \geq 0$ such that ${\mathcal E}_{n,S} \ne \emptyset$, we have
\[
S \geq (n-n_{\mbox{\tiny hole}}-n_\delta)/\tilde \theta.
\]
In other words, we disregard the hole free time $n_{\mbox{\tiny hole}}$.
The proof is basically the same as in \cite{DHL} if we keep in mind that at an
essential return to $c \in \Crit_c \cup \Crit_s$
at time $\ell$, following a hole return,
the size of the interval $\T^\ell(\omega')$ that emerges from the cut
at this essential return depends only on the distance of $\T^\ell(\omega')$
to $c$.

Now Lemmas 8 and 9 of \cite{DHL} and \eqref{eq:21} combine to give
for ${\mathcal E}_n(\omega) := \cup_{S \geq 0} {\mathcal E}_{n,S}(\omega)$:
\begin{eqnarray*}
| {\mathcal E}_n(\omega) | &=&
\sum_{0 \leq n_{\mbox{\tiny hole}} \leq n}
\ \ \sum_{S \geq (n-n_{\mbox{\tiny hole}} - n_\delta)/\tilde\theta}
| {\mathcal E}_{n,S,}(\omega)| \\
&\leq & \sum_{0 \leq n_{\mbox{\tiny hole}} \leq n} e^{-\lambda_{\mbox{\tiny hole}} n_{\mbox{\tiny hole}}} \sum_{S \geq (n-n_{\mbox{\tiny hole}} - n_\delta)/\tilde\theta}
\kappa^{-1} e^{-(\theta^*-\tilde \eta) S} \\
& \leq & \tilde C_1 e^{(\theta^*-\tilde \eta) n_\delta/ \tilde \theta}
e^{ - \min\{ (\theta^*-\tilde \eta)/\tilde \theta\ , \ \lambda_{\mbox{\tiny hole}} \} n},
\end{eqnarray*}
for some $\tilde C_1$ as in  the formula given near the bottom of \cite[page
444]{DHL}.\footnote{with the factor $\kappa^{-1}$ inserted where it is missing in
\cite{DHL}.}

Having established Lemmas 6 to 9,  the rest of the proof in \cite{DHL}
goes through basically unchanged since the decay in $\hat R^{(\tH)}$
depends only on the rates in these
lemmas and distortion estimates which are not affected by $\tH$.
We see that $\theta(\tH)$ can be made arbitrarily close to $\theta_0$ for
$h$ sufficiently small.
\end{proof}

We have shown that in the presence of additional cuts introduced by
$\partial \tH$, we retain some uniform control over the induced towers
$(\F^{(\tH)}, \hDelta^{(\tH)})$.  We are now ready to lift the holes
into the towers and consider the open systems so defined.

We define the hole in the tower and the return time with hole to be
\[
H = \{ Z \in \Z : \pi Z \subseteq \tH \}
\qquad \mbox{and} \qquad
R(x) = \min \{ \hat R(x), \min\{ j : \T^j(x) \in \tH \} \}.
\]
For any partition element $Z_{\ell,j} = H_{\ell,j}$ that is identified
as a hole,
we delete all levels in the tower above $Z_{\ell,j}$ since nothing is mapped
to those elements once the hole is introduced.  We denote the remaining
tower with holes by $\Delta$ and define $F = \F|_{\Delta \cap \F^{-1} \Delta}$
to be the corresponding tower map.

In order to invoke the conclusions of Proposition~2.6 for $(F, \Delta)$,
we must check that its hypotheses, (P1)-(P3) and (H1), are satisfied.
We then check conditions (A1) and (A2) in order to project the
convergence results from the tower to the underlying system.

Properties (P1)-(P3) are
automatic for $(F, \Delta, R)$ since they hold for $(\F, \HDelta, \hat R)$.

\medskip
\noindent
{\em Step 1. Condition (H1).}
We split the sum in (H1) into pieces that
encounter $\tH$ during their bound period and those that encounter it
when they are free,
\[
\sum_\ell m(H_\ell) \beta^{-\ell} = \sum_{\mbox{\scriptsize bound}} m(H_\ell) \beta^{-\ell}
+ \sum_{\mbox{\scriptsize free}} m(H_\ell) \beta^{-\ell} .
\]

To estimate the bound pieces, we use the slow recurrence condition given by
(B1).
If $\omega \subset I^*$ is some partition element, and
 $c \in \Crit_c$ the last critical point visited by $\omega$ before
$\omega$ falls into the hole, then
dist$(\T^\ell c, \partial \tH_j) \geq \delta e^{-\alpha^*_c \ell}$ for each
$j$.
Therefore, if $\T^\ell \omega \cap \tH_j \neq \emptyset$,
we must have $\delta e^{-\alpha^*_c \ell} < \tm (\tH_j)$ and so
$\ell > - (1/\alpha^*_c) \log (\tm(\tH_j)/\delta)$.  Thus by Property (P1),
\begin{eqnarray}
\label{eq:bound escape}
\sum_{\mbox{\scriptsize bound}} m(H_\ell) \beta^{-\ell}
&\leq& \sum_{\ell > - \frac1{\alpha^*_c} \log (\tm(\tH)/\delta)} m(\HDelta_\ell) \beta^{-\ell} \nonumber \\
&\leq& \sum_{\ell > - \frac1{\alpha^*_c} \log (\tm(\tH)/\delta)} C \theta^\ell \beta^{-\ell}
\leq \frac{C'}{\delta} \, \tm(\tH)^{\frac1{\alpha^*_c} \log(\theta^{-1} \beta)}.
\end{eqnarray}

To estimate the free pieces, we will need some facts about the tower without
holes, $(\F, \HDelta)$.
It was shown in \cite{young} that $\F$ admits a unique absolutely continuous invariant measure $\eta$ with density $\rho \in \B_0$, $\rho \geq a >0$.  Moreover, $\pi_* \eta = \teta$ is the unique
SRB measure for $\T$.  By Section~4.1, $\trho = \ppi \rho$ is the density of
$\teta$.

Notice that $\rho|_{\HDelta_0}$ is an invariant density for $\F^{\hat R}$
so that $\trho_0 := \ppi (\rho_{\HDelta_0})$ is an invariant density for $\T^{\hat R}$.
Since $\pi' \equiv 1$ on $\HDelta_0$, we have $a \leq \trho_0 \leq A$.
This implies that we
can also obtain the invariant density $\trho$ by {\em pushing forward} $\trho_0$
under iterates of $\T$.

It is clear that pushing forward $\trho_0$ will result in spikes above the orbits of
the critical points, hence $\trho$ is not bounded on $\I$.  However, when an
interval $\omega \subset \pi(\HDelta_0)$ is free at time $n$,
condition (C2) and \cite[Lemma 1]{DHL} imply that
the push forward of the density on $\omega$ at time $n$ will be uniformly bounded.

Define neighborhoods $N_k(\T^kc)$ of radius
$\delta e^{-2\alpha^*_ck}$ for each $c \in \Crit_c$.  These are precisely the
points starting in $B_\delta(c)$ whose orbits are still bound to $c$ at time $k$.
From the above considerations, it is clear that outside of the set
$\cup_{c \in \sCrit_c} \cup_{k \geq 1} N_k(\T^kc)$, the density $\trho$ is bounded.
{\em This is the sum of the push forwards of $\trho_0$ on free pieces.}  Thus,
we may define a measure
\[
\teta_{\mbox{\scriptsize free}} = \sum_{(\ell,j): Z_{\ell,j} \ \mbox{\scriptsize is free}}
                \pi_*\eta(Z_{\ell,j})
\]
whose density with respect to Lebesgue, $\trho_{\mbox{\scriptsize free}}$,
is bounded on $\I$.
Then since $\rho \geq a > 0$, we have
\[
\sum_{\mbox{\scriptsize free}} m(\hlj) \leq \sum_{\mbox{\scriptsize free}}  \eta(H_{\ell,j})/a = \teta_{\mbox{\scriptsize free}}(\tH)/a
\leq C \tm(\tH).
\]
Now set $P= - \log \tm (\tH)$.
We estimate the contribution from free pieces by
\begin{eqnarray}
\sum_{\mbox{\scriptsize free}} m(\hlj) \beta^{-\ell} &
= & \sum_{\mbox{\scriptsize free: } \ell > P} m(\hlj) \beta^{-\ell}
+ \sum_{\mbox{\scriptsize free: } \ell \leq P} m(\hlj) \beta^{-\ell}   \nonumber  \\
& \leq & \sum_{\mbox{\scriptsize free: } \ell > P} C \theta^\ell \beta^{-\ell}
+ \beta^{-P} \sum_{\mbox{\scriptsize free: } \ell \leq P} m(\hlj)  \nonumber \\
& \leq & C' (\theta \beta^{-1})^P + C'' \beta^{-P} \tm(\tH) \nonumber \\
& \leq & C' \tm(\tH)^{\log (\beta \theta^{-1})} + C'' \tm(\tH)^{1+ \log \beta}  .
\label{eq:free escape}
\end{eqnarray}

Putting together \eqref{eq:bound escape} and \eqref{eq:free escape}, we see that
the left hand side of (H1) is proportional to
$\tm(\tH)^\gamma$, for some $\gamma > 0$.  This quantity
can be made sufficiently small to satisfy (H1) by choosing $\tm(\tH)$ small
since $\theta$ (and hence $\beta$) are independent of $\tH$ by
Lemma~\ref{lemma:uniform decay}.

\medskip
\noindent
{\em Step 2. Property (A1).}
The bounded distortion required by (A1)(b) is satisfied by the cutting of pieces
introduced in the construction of $\Delta$ (see \cite[Proposition 3]{DHL}).
The expansion
required by (A1)(a) follows from two estimates:
property (C1) guarantees that
starting at any $x \notin B_\delta(\Crit)$, there is exponential expansion upon entry to
$B_\delta(\Crit)$; \cite[Lemma 2]{DHL} guarantees that exponential expansion occurs
at the end of a binding period.  Since any return must occur at a free entry to
$B_\delta(c^*)$, we may concatenate these estimates as many times as needed
in order to obtain (A1)(a) at any return time $\hat R_n$.  However, once the
hole is introduced, a partition element may fall into the hole during a bound period
and so the return time with hole, $R$, may be declared when there has not been
sufficient expansion to satisfy (A1)(a).
Since this property is only needed
to prove Lemma 4.1, we give an alternate proof of this lemma which uses
(A1)(a) only for $\hat R$.
\medskip

\begin{proof}[Proof of Lemma 4.1 for $(F,\Delta)$]
First note that because (A1) is satisfied by $(\F, \HDelta)$,
Lemma~4.1
holds for lifts $\tf \circ \pi$ of $\tf \in C^{\bar \alpha}(\I)$, with
$\bar \alpha \geq - \log \beta/ \log \tau$.  Here $\tau$ is the rate of expansion from (A1)
and $\beta$ is the constant chosen for the symbolic metric on $\HDelta$
(see Section 2.1.1).

The separation time $\hat s(\cdot,\cdot)$ is shortened
by the introduction of the hole in the tower so that the new separation time
satisfies $s(x,y) \leq \hat s(x,y)$.  Thus the separation
time metric is also loosened on $\Delta$:
\begin{equation}
\label{eq:loose metric}
d_\beta(x,y) := \beta^{s(x,y)} \geq \beta^{\hat s(x,y)} =: \hat d_\beta(x,y)  .
\end{equation}
Thus if $f$ is Lipschitz with respect to $\hat d_\beta$ on $\HDelta$,
its restriction to $\Delta$ is also Lipschitz with respect to $d_\beta$.

Now for $\tf \in C^{\bar \alpha}(\I)$, with $\bar \alpha \geq - \log \beta/ \log \tau$, we have
$\tf \circ \pi \in \B_0(\HDelta)$ by Lemma~4.1.
Then by \eqref{eq:loose metric}, $\tf \circ \pi \in \B_0(\Delta)$ as well.
\end{proof}

\medskip
\noindent
{\em Step 3. Property (A2).}
We focus first on finding an index set $J \subset \N \times \N$ such that
(A2)(a) is satisfied.  The following lemma is the analogue of \eqref{eq:no holes mixing}
for $T$, the map with holes. (See also \cite[Lemma 5.2]{demers logistic}.)

\begin{lemma}
\label{lemma:growth}
Let $\delta$ be the radius of $B_\delta(c^*)$ as above.
Let $n_0 = n(\delta)$ be defined by
\eqref{eq:no holes mixing}.  For $h$ sufficiently small,
given any interval $\omega \subset I$ such that $|\omega| \geq \delta/3$,
we have
\[
\bigcup_{i=0}^{2n_0} T^i \omega \supset I \quad  \bmod{0}
\]
\end{lemma}

\begin{proof}
Suppose there exists an interval $A$ such that
$A \cap (\cup_{i=0}^{n_0} T^i \omega) = \emptyset$.
Since $A \subseteq \T^{n_0} \omega$, we must have
$A \cap \T^{i_k} \tH_k \neq \emptyset$ for some $\tH_k$
such that $\tH_k \cap \T^{i_k'} \omega \neq \emptyset$ for some integers
$i_k, i_k'$ with $i_k + i_k' = n_0$.
In other words, the piece of
$\omega$ that should have covered part of $A$ fell into $\tH_k$
before time $n_0$.

Condition (B2)(b) implies that there exists $1 \leq j_k \leq k$ such that
\[
\min_{1 \leq \ell \leq n_0} \mbox{dist}(g_{k,j_k}, \T^\ell b_k)  > 0 .
\]
Thus for small $h$, we have $\T^\ell(\tH_k) \cap B_h(g_{k,j_k}) = \emptyset$
for all $1 \leq \ell \leq n_0$.  So $B_h(g_{k,i})$ is covered by time $n_0$
under $T$, i.e., $B_h(g_{k,j_k}) \subset T^{n_0} \omega$.

Since $\T(b_k) = T(g_{k,j_k})$, condition (B2)(a) says that $B_h(g_{k,j_k})$
cannot fall into the hole before time $n_0$ for small $h$.  Thus
$T^{i_k}B_h(g_{k,j_k}) \supseteq \T^{i_k}\tH_k$ and we conclude that
the part of $A$
which should have been covered by the piece of $\omega$ that
fell into $\tH_k$ is at the latest covered at time $n_0 + i_k$ by an interval
passing through $B_h(g_{k,j_k})$.

Doing this for each $k$, we have
$A \subset \cup_{k=1}^L T^{i_k}B_h(g_{k,j_k})$ and so
$A \subset \cup_{i=0}^{2n_0} T^i \omega$.
\end{proof}

\medskip
Lemma~\ref{lemma:growth} implies that $I$ can be covered by the
projection of finitely many levels of $\Delta$, say the first $N$.
If $\pi(Z_{\ell,j}) \subset \pi(Z_{\ell',j'})$ and both $\ell, \ell' \leq N$,
we eliminate $(\ell,j)$ from our index set, but retain $(\ell',j')$.
By \eqref{eq:disjoint}, the remaining index set
$J \subset \{0, \ldots, N-1\} \times \N$ satisfies (A2)(a) and (A2)(b).
As before, set $\Z_J = \cup_{(\ell,j) \in J} \zlj$.

By \eqref{eq:pi}, $J\pi(F^\ell x) = (T^\ell)'(\pi x)$ so that we are only concerned
with the first $N$ iterates of $T^\ell$.  It is clear that if Crit$_s = \emptyset$
and $T$ is globally $C^2$, then $J\pi|_{\Z_J} \in \B_0$ by
Lemma~\ref{lem:lift} and so (A2)(c) is satisfied.

In the case when Crit$_s$ is nonempty,
(A2)(c) does not hold and so Proposition~\ref{prop:continuous} must be modified.
We do this in Step 5 of the proof when we address the convergence
property for the \accim.

\medskip
\noindent
{\em Step 4.  $(F,\Delta)$ is mixing.}
Since we have constructed a tower over a single base, it suffices to show
that g.c.d.$(R)=1$.
The fact that $\T$ is nonrenormalizable guarantees that for the
infinitesimal hole $\tH^{(0)}$, $(\F^{(0)}, \Delta^{(0)})$ can be
constructed to be mixing by making g.c.d.$(\hat R^{(0)}) = 1$.
Indeed, \eqref{eq:no holes mixing}
implies that as in the
case of expanding maps, we can simply wait one time step
on a given return to destroy any periodicity in $\hat R^{(0)}$.
Once this is accomplished, Lemma~\ref{lemma:good returns} implies that
$(\F^{(\tH)}, \hDelta^{(\tH)})$
is mixing for $\tH \in \Ho(h)$ with $h$ small enough that g.c.d.$(\hat R^{(\tH)})$
is still $1$ (by making $n_h$ sufficiently large).  But since holes
cannot affect returns before level $n_h$ in $\Delta$, the tower with holes,
we have that g.c.d.$(R)=1$ as well.

\medskip
\noindent
{\em  Step 5.  Convergence property.}
We have already verified in Steps 2 and 3 that $(F,\Delta)$
satisfies (A1) and there is an index set $J$ satisfying (A2)(a) and (A2)(b).

By \eqref{eq:pi}, the problem spots where $(T^\ell)'$ (and therefore $J\pi$) are unbounded
are neighborhoods of $T^k(c)$  for $c \in \Crit_s$, $k\geq 1$.  In fact,
we only need to address the iterates of $c \in \Crit_s$ up until the time when
a neighborhood of $T^k(c)$ is covered by some other element in the tower
on which the derivative is bounded.  Since $I$ is covered by the first $N$ levels of
$\Delta$, we need consider
at most the first $N$ iterates of $c \in \Crit_s$.

Notice that if a neighborhood $A$ of $T^k(c)$ can only be reached
by an interval
$\omega$ originating in a neighborhood of $c$, then due to the exponential partition
of $B_\delta(c)$ which subdivides $\omega$, there are countably many
elements $Z \subset \Delta_\ell$ whose projections cover $A$ and in which $|\pi'|$ becomes unbounded the closer that $\pi Z$ is to $T^k(c)$.

Fix $\ve >0$ and let $\Na_\ve(c)$ denote the $\ve$-neighborhood of those iterates of $c
\in \Crit_s$ which can only be reached by passing through $B_\delta(c)$. Let $\Na_\ve =
\cup_{c \in \sCrit_s} \Na_\ve(c)$ and let $J_1 \subset J$ be the index set of those
elements $Z$ such that $\pi Z \subset \Na_\ve$. Denote by $1_\ve$ the indicator function
of the set $\{ y \in I : y \in \pi Z_{\ell,j}$, $(\ell,j) \in J_1\}$.

Now let $\tf \in \G$ and write $\tf = \tf_0 + \tf_\ve$ where
$\tf_\ve := \tf \cdot 1_\ve$ and $\tf_0 = \tf - \tf_\ve$.  We define a lift of $\tf$ by
$f = \tf \circ \pi \cdot J\pi$
on elements of $J$ as in the proof of Proposition~\ref{prop:continuous}.
The lifts $f_0$ and $f_\ve$ are defined analogously.
Although $f \notin \B$, we do have $f_0 \in \B_0$ by Proposition~\ref{prop:continuous}
since $\tf_0 \circ \pi \equiv 0$ on those elements in which $J\pi$ becomes
unbounded.  Using Corollary~\ref{cor:convergence} precisely as in
Section~\ref{expanding},
we have
\begin{equation}
\label{eq:0-conv}
\lim_{n\to \infty} \lambda^{-n} \Lp^n \tf_0 = c(\tf_0) \tp
\end{equation}
where convergence is in the $L^1$-norm and $c(\tf_0)>0$.  Since
\[
\lambda^{-n} \Lp^n \tf = \lambda^{-n} \Lp^n \tf_0 + \lambda^{-n} \Lp^n \tf_\ve ,
\]
our strategy will be to show that the $L^1$-norm of the second term above
can be made uniformly small
in $n$ by making $\ve$ small.  This will imply that $\lambda^{-n} \Lp^n \tf \to c(\tf) \tp$
in $L^1(\tm)$
where $c(\tf) = \lim_{\ve \to 0} c(\tf_0) > 0$, implying the desired convergence result.

Estimating $|\Lp^n \tf_\ve|_{L^1(\tm)}$ is equivalent to estimating $|\Lp^n
f_\ve|_{L^1(m)}$.
\begin{equation}
\label{eq:int est} \lambda^{-n} \int \Lp^n f_\ve \, dm = \lambda^{-n} \sum_{(\ell, j)
\in J_1} \int_{\Delta^n \cap \zlj}  f_\ve \, dm \leq \lambda^{-n} \sum_{(\ell,j) \in
J_1} |\tf|_\infty |J\pi_{\ell,j}|_\infty m(\Delta^n \cap \zlj).\end{equation} Since we
are concerned with finitely many problem spots where the derivative blows up, it
suffices to show that the sum in \eqref{eq:int est} over elements in $J_1$ corresponding
to one of the problem spots is proportional to $\ve$. For simplicity, we fix $c \in
\Crit_s$ and denote by $\A_\ve$ the set of elements in $\Delta$ projecting to the
$\ve$-neighborhood of $\T(c)$.  There exists $k_\ve >0$ such that if $Z \in \A_\ve$,
then $\pi Z$ lies in an element of the partition $E^-_k = \T(c - e^{-k+1}, c - e^{-k})$
and $E^+_k = \T(c + e^{-k}, c + e^{-k+1})$ with $k \geq k_\ve$.

For $\zlj \in \A_\ve$, let $Z_{0,j} = F^{-\ell}\zlj$.  We split the sum in \eqref{eq:int est} into those elements $Z \in \A_\ve$ with
$R(Z) \geq n$ and those with $R(Z) < n$.  Let $0<\ell_c<1$ denote the critical
order of $c$.

We estimate terms with $R(Z)<n$ using Corollary~\ref{cor:uniform cyl}
and the bounded distortion given by (A1)(b) for $J\pi$.
\begin{equation}
\label{eq:R<n}
\begin{split}
&\lambda^{-n}  \sum_{\zlj \in \A_\ve : R(\zlj) < n} |J\pi_{\ell,j}|_\infty  m(\Delta^n \cap \zlj)
\leq \sum_{Z_{\ell,j} \in \A_\ve : R(Z_{\ell,j}) < n} |J\pi_{\ell,j}|_\infty C \lambda^{-R(Z_{\ell,j})} m(Z_{\ell,j}) \\
& \leq C' \int_{\A_\ve:R(z)<n} J\pi \, \lambda^{-R} \, dm
\leq C' \Big( \int_{\A_\ve:R(Z)<n} (J\pi)^p \, dm \Big)^{1/p}
\Big( \int_{\A_\ve:R(Z)<n} \lambda^{-Rp/(p-1)} \, dm \Big)^{1-1/p}
\end{split}
\end{equation}
where $1 < p < 1/(1-\ell_c)$.
By \eqref{eq:pi}, we have
$J\pi_{\ell,j} = |(T^\ell)'|_{Z_{0,j}}|$ so that if $\pi(\zlj) \subset E^\pm_k$, we have
$J\pi_{\ell,j} \approx e^{k(1-\ell_c)}$.  Also, $J\pi_{\ell,j}$ has bounded distortion
across all $\zlj$ that project into a single $E^\pm_k$.  So we estimate the first factor in
\eqref{eq:R<n} by
\begin{equation}
\label{eq:jpi bound}
\begin{split}
\sum_{\zlj \in \A_\ve:R(Z)<n} |J\pi_{\ell,j}|^p_\infty m(Z_{\ell,j})
& \leq \sum_{k \geq k_\ve} \sum_{\pi Z \subset E_k^\pm} C e^{k(1-\ell_c)p} m(Z_{0,j}) \\
& \leq \sum_{k \geq k_\ve} C e^{k(1-\ell_c)p} e^{-k}
\leq C e^{-k_\ve(1-p(1-\ell_c))}  .
\end{split}
\end{equation}
To estimate the second factor in \eqref{eq:R<n}, notice that if $\pi Z \subset E^\pm_k$,
then $R(Z)>\log k$.  Thus
\begin{equation}
\label{eq:lambda bound}
\begin{split}
\sum_{\zlj \in \A_\ve:R(Z)<n} \lambda^{-R(Z)p/(p-1)} m(Z)
& \leq \sum_{r>\log k_\ve} \sum_{R(Z)=r}  \lambda^{-rp/(p-1)} m(Z) \\
& \leq C \sum_{r>\log k_\ve} (\lambda^{- p/(p-1)} \theta)^r
\leq C' (\lambda^{- p/(p-1)} \theta)^{\log k_\ve} .
\end{split}
\end{equation}
To estimate the terms of \eqref{eq:int est} with $R(Z) \geq n$, notice that
for such $Z$, $\Delta^n \cap Z = Z$.
\[
\begin{split}
\lambda^{-n} \sum_{Z \in \A_\ve:R(Z) \geq n} |J\pi_{\ell,j}|_\infty m(Z)
& \leq C \lambda^{-n} \sum_{k \geq k_\ve} e^{k(1-\ell_c)}
\sum_{\pi Z \subset E^\pm_k: R(Z) \geq n} m(Z) \\
& \leq C \lambda^{-n} \sum_{k \geq k_\ve} e^{k(1-\ell_c)}
m(\A_\ve \cap \pi^{-1}E^\pm_k \cap \{R\geq n\}).
\end{split}
\]
We let $A_{k,n} = \A_\ve \cap \pi^{-1}E^\pm_k \cap \{R\geq n\}$.  On the one hand,
since $R>n$ on $A_{k,n}$, we have $m(A_{k,n}) \leq C\theta^n$;
on the other hand, $m(A_{k,n}) \leq e^{-k}$ by definition of the partition.  Choose
$0 < \gamma < \ell_c$ and write
$m(A_{k,n}) = m(A_{k,n})^\gamma m(A_{k,n})^{1-\gamma}$.
Then
\begin{equation}
\label{eq:R>n}
\lambda^{-n} \sum_{Z \in \A_\ve:R\geq n} |J\pi_{\ell,j}|_\infty m(Z)
\leq C \lambda^{-n} \sum_{k \geq k_\ve} e^{k(1-\ell_c)} \theta^{n\gamma}
e^{-k(1-\gamma)}
\leq C' (\lambda^{-1}\theta^\gamma)^n e^{-k_\ve(\ell_c - \gamma)}
\end{equation}
Putting together \eqref{eq:jpi bound}, \eqref{eq:lambda bound} and \eqref{eq:R>n} we see
that \eqref{eq:int est} becomes
\begin{equation}
\label{eq:first conv}
\lambda^{-n} \int \Lp^n f_\ve \, dm
\leq C e^{-k_\ve(1 - p(1-\ell_c))/p} (\lambda^{-1}\theta^{(p-1)/p})^{\log k_\ve}
+ C'  (\lambda^{-1}\theta^\gamma)^n e^{-k_\ve(\ell_c - \gamma)}  .
\end{equation}
When the holes are sufficiently small, i.e., when
$\lambda^{-1} \geq \max \{ \theta^\gamma, \theta^{(p-1)/p} \}$, this
quantity can be made arbitrarily small independently of $n$.

\medskip
\noindent
{\em Step 6. Exponential rate of convergence.}  We show that the
convergence of $\lambda^{-n} \Lp^n \tf$ established in Step 5
occurs at an exponential rate.  Since $|\Lp^n \tf|_{L^1(\tm)} = |\Lp^n f|_{L^1(m)}$,
it suffices to show this convergence for the lift on $\Delta$.

Let $\ve = e^{-tn}$ for some small constant $t$ to be chosen later.
Define $\Na_\ve$ as above and notice that outside of $\Na_\ve$,
the $C^2$ norm of $\T^\ell$ for $\ell =1, \ldots, N$ is proportional to
$e^{-k_\ve(\ell_c^* -2)}$ where $\ell_c^{\mi}>0$ is the minimum of the critical
orders of $c \in \Crit_s$.  Since $k_\ve$ is on the order of $- \log \ve$,
we have $|\T^\ell|_{\I \backslash \Na_\ve}|_{C^2} = \mathcal O(e^{tn(2- \ell_c^{\mi})})$.
Let $\Z_J = \cup_{(\ell,j) \in J} \zlj$ and let $\Z_{J,\ve} \subset \Z_J$ denote those
elements which project into $\Na_\ve$.  Then
\begin{equation}
\label{eq:norm growth}
\|J\pi|_{\Z_J\backslash \Z_{J,\ve}}\|_0 \leq C e^{tn(2-\ell^{\mi}_c)} .
\end{equation}
Define $\tf_0$, $f_0$, $\tf_\ve$, $f_\ve$ as in Step 5.
By Lemma~\ref{lem:lift}, \eqref{eq:norm growth} implies
that $\| f_0 \| \leq C e^{tn(2-\ell^{\mi}_c)}$ so that by Corollary~\ref{cor-exp},
\begin{equation}
\label{eq:0 part}
|\lambda^{-n} \Lp^n f_0 - c(f_0) \vf |_{L^1(m)}
\leq \| \lambda^{-n} \Lp^n f_0 - c(f_0) \vf \| \leq C e^{tn(2-\ell^{\mi}_c)} \sigma^n
\end{equation}

Next, when the holes are sufficiently small,
$\lambda^{-1} \geq \max \{ \theta^\gamma, \theta^{(p-1)/p} \}$, so
\eqref{eq:first conv} yields,
\begin{equation}
\label{eq:ve part}
\lambda^{-n} |\Lp^n f_\ve|_1 \leq C e^{-tn(1-p(1-\ell^{\mi}_c))/p} +
C'e^{-tn(\ell^{\mi}_c - \gamma)} \leq C'' e^{-tn \gamma'}
\end{equation}
for some $\gamma'>0$.  In particular, we see from \eqref{eq:ve part} that
the constants $c(f_0)$ converge to $c(f)$ exponentially fast as well.
\[
|c(f) - c(f_0)| = \lim_{n\to \infty} \lambda^{-n} (|\Lp^nf|_1 - |\Lp^nf_0|_1)
= \lim_{n\to\infty} \lambda^{-n} \int \Lp^n f_\ve \, dm \leq C'' e^{-tn \gamma'}
\]
This estimate together with \eqref{eq:0 part} and \eqref{eq:ve part} imply that
$\lambda^{-n} \Lp f \to c(f) \vf$ exponentially fast once we choose $t < -\log
\sigma/(2-\ell^{\mi}_c)$.

\medskip
\noindent
{\em  Step 7.  Unified escape rate.}
By Step 5, for each $\tf \in \G$, we have $f \in L^1(m)$ such that
$\ppi f = \tf$ and $\lambda^{-n} \Lp_F^n f = c(f) \vf$ for some $c(f)>0$
which implies
$\lambda^{-n} \Lp_T^n \tf = c(f) \tp$ by \eqref{eq:conv in X}.
 Letting $\teta = \tf \tm$, we have
\[
\lim_{n\to \infty} \frac{1}{n} \log \teta(I^n) =
\lim_{n\to \infty} \frac{1}{n} \log |\Lp_T^n \tf|_1 = \log \lambda  .
\]
\end{proof}

%%%%%%%%%%%%%%%%%%%%%%%%%%%%%%%%%%%%%

\subsubsection{Small hole limit}
\label{small hole}

\begin{proof}[Proof of Theorem~\ref{thm:small hole}]
By \eqref{eq:bound escape} and \eqref{eq:free escape},
the quantity $q := \sum_{\ell \geq 1} m(H_\ell) \beta^{-(\ell-1)}$ can be made arbitrarily
small by choosing $h$ to be small.  By Proposition~\ref{prop:regularity},
the escape rate $\lambda$
is controlled by the size of $q$ so that $\lambda \to 1$ as $q \to 0$.  Thus
$\lambda_h \to 1$ as $h \to 0$.

Since $\tmu_h$ is a sequence of probability measures on the compact interval $\I$,
it follows that a subsequence, $\{ \tmu_k \}$ corresponding to $h_k$,
converges weakly to a probability measure
$\tmu_\infty$.  We show that $\tmu_\infty$ is an absolutely continuous invariant
measure for $\T$.  Since there is only one such measure, this will imply that in
fact the entire sequence converges to this same invariant measure.

\medskip
\noindent
{\em Step 1.  $\tmu_\infty$ is absolutely continuous with respect to Lebesgue.}
For each $\tH^{(k)}$, we have two towers:  $(\F^{(k)}, \hDelta^{(k)})$ which has no holes
but is constructed using $\partial \tH^{(k)}$ as artificial cuts as described in the proof
of Theorem~\ref{thm:CE convergence}; and $(F^{(k)}, \Delta^{(k)})$, the tower with
holes obtained from $\hDelta^{(k)}$.
By Lemma~\ref{lemma:uniform decay}, there exist uniform constants
$C > 0$, $\theta < 1$ such that $m(\hDelta^{(k)}_\ell) \leq C\theta^\ell$.

We have an invariant density $\rho_k$ on $\hDelta^{(k)}$ and a conditionally invariant
density $\vf_k$ on $\Delta^{(k)}$.
By Proposition~\ref{prop:regularity}, both $\rho_k, \vf_k \in \B_M$ where
$M$ is independent of $k$  (to see the results for $\rho_k$, simply apply
the proposition to the case $H = \emptyset$).  In addition,
by Proposition~\ref{prop:spectral picture}(i), $\rho_k \geq a>0$ and the
constant $a$ is independent of $i$ because the uniform decay
given by Lemma~\ref{lemma:uniform decay} implies that $\hDelta_0^{(k)}$ must
retain some positive minimum measure for all $k$.

Let $\hat \pi_k$ be the projection corresponding to $\hDelta^{(k)}$ and let
$\pi_k = \hat \pi_k|_{\Delta^{(k)}}$.  Letting $\trho$ denote the unique invariant
density for $\T$ and $J\hat\pi_k$ the Jacobian of $\hat\pi_k$ etc., we have for each $k$,
\begin{equation}
\label{eq:inv proj}
\trho(x) = {\mathcal P}_{\hat \pi_k} \rho_k(x) = \sum_{y \in \hat \pi_k^{-1}x} \frac{\rho_k(y)}{J\hat \pi_k(y)}
\; \; \; \; \mbox{and} \; \; \; \;
\tp_k(x) := {\mathcal P}_{\pi_k} \vf_k(x) \sum_{y \in \pi_k^{-1}x} \frac{\vf_k(y)}{J\pi_k(y)}.
\end{equation}
Now for any $\ve>0$, choose $L > 0$ such that
$\sum_{\ell > L} CM\beta^{-\ell} \theta^\ell < \ve$.  Next choose $k_0$ such that
for all $k \geq k_0$, $\lambda_k^{-L} \leq 2$.  Now
for any Borel $A \subset \I$,
\[
\tmu_k (A) = \sum_{\ell \leq L} \mu_k(\Delta_\ell \cap \pi_k^{-1}A)
+ \sum_{\ell > L} \mu_k(\Delta_\ell \cap \pi_k^{-1}A)
=: \tmu_{k,L}(A) + \tmu_{k,+}(A) .
\]
By \eqref{eq:inv proj}, the measure
$\tmu_{k,L}$ has density $\tp_{k,L}$ bounded independently of $k \geq k_0$:
\begin{equation}
\label{eq:L}
\tp_{k,L}(x) = \sum_{y \in \pi_k^{-1}x: \ell(y) \leq L} \frac{\vf_k(y)}{J\pi_k(y)}
\leq \frac{M}{a} \sum_{y \in \hat \pi_k^{-1}x: \ell(y) \leq L} \lambda_k^{-\ell(y)}
\frac{\rho_k(y)}{J\hat \pi_k^{-1}(y)}
\leq \frac{2M \trho(x)}{a}.
\end{equation}
where $\ell(y)$ is the level of $y$ in $\hDelta^{(k)}$.
The remaining measure $\tmu_{k,+}$ has small total mass:
\begin{equation}
\label{eq:+}
\tmu_{k,+}(\I) = \sum_{\ell > L} \tmu_k(\Delta_\ell)
\leq \sum_{\ell > L} M \beta^{-\ell} m(\Delta_\ell)
\leq \sum_{\ell > L} CM \beta^{-\ell} \theta^{\ell} < \ve .
\end{equation}
Putting together \eqref{eq:L} and \eqref{eq:+}, we see that
$\mu_\infty = \mu_{\infty,L} + \mu_{\infty, +}$ where $\mu_{\infty, L}$ has
density bounded by $2M \trho/a$ while $\mu_{\infty, +}$ is possibly singular
with total mass less than $\ve$.  Since this is true for each $\ve >0$, we conclude
that in fact $\mu_\infty$ is absolutely continuous with density bounded
by $2M \trho/a$.

\medskip
\noindent
{\em Step 2.  $\tmu_\infty$ is invariant.}
Let $I_k = \I \backslash \tH^{(k)}$ and as usual, let $I^n_k = \cap_{j = 0}^n \T^{-j}I_k$
and $T_k = \T|_{I^1_k}$.

By Step 1, $\tmu_\infty$ has density bounded by $2M \trho/a$, which is in
$L^1(\tm)$.  Thus $\tmu_\infty$ gives 0 measure to the singularity points of $\T$.
This fact allows us to write, for any continuous function $f$ on $\I$,
\begin{equation}
\label{eq:invariant}
\tmu_\infty(f \circ \T) = \lim_{k \to \infty} \tmu_k (f \circ \T)
= \lim_{k \to \infty} \int_{I^1_k} f \circ T_k \, d\tmu_k
+ \int_{\I \backslash I^1_k} f \circ \T \, d\tmu_k .
\end{equation}
Since $\lambda_k \to 1$, the first term in \eqref{eq:invariant} is equal to
\[
\lim_{k \to \infty} \int_{I_k} f  \, d((T_k)_*\tmu_k)
= \lim_{k \to \infty} \lambda_k \tmu_k (f) = \tmu_\infty(f) .
\]
The second term in \eqref{eq:invariant} is
bounded by
$|f|_\infty \, \tmu_k (\I \backslash I^1_k)$.  This quantity tends to 0 as
$k \to \infty$ because of the uniform bounds on the densities of $\tmu_k$
obtained in Step 1.
\end{proof}

%%%%%%%%%%%%%%%%%%%%%%%%%%%%%%%%%%%%%%%%%%%

\section{Equilibrium Principle}
\label{equilibrium}

In this section we consider the invariant measures $\nu$ and $\tnu = \pi_*\nu$
and prove Theorems \ref{thm:F-equilibrium},
\ref{thm:inverse limit} and \ref{thm:T-equilibrium}.
We assume throughout that $F$ is mixing and satisfies (P1)-(P3) and (H1).

%%%%%%%%%%%%%%%%%%%%%%%%%%%%%%%%%%

\subsection{Characterization of $\tnu$}
\label{inverse limit}

\begin{proof}[Proof of Theorem~\ref{thm:inverse limit}]
Let $\tf \in C^{\bar \alpha}(X)$ and
note that $\tf \circ \pi \in \B_0$.  Thus,
\begin{eqnarray*}
\tnu(\tf) & = & \nu(\tf \circ \pi) \; \;
  = \; \; \lim_{n\to \infty} \lambda^{-n}
        \int_{\Delta^n} \tf \circ \pi \, d\mu \\
  & = & \lim_{n\to \infty} \lambda^{-n}
         \int_{\Delta^n} \tf \circ \pi \vf \, dm
  \; \; = \; \; \lim_{n\to \infty} \lambda^{-n} \int_{\pi(\Delta^n)}
              \ppi(\tf \circ \pi \, \vf) \, d\tm \\
  & = & \lim_{n\to \infty} \lambda^{-n} \int_{X^n} \tf \tp \, d\tm
  \; \; = \; \; \lim_{n\to \infty} \lambda^{-n} \int_{X^n} \tf \, d\tmu
\end{eqnarray*}
where in the first line we have used Proposition~\ref{prop:inv measure}.

The ergodicity of $\tnu$ follows from that of $\nu$
and the relation $X^\infty = \pi(\Delta^\infty)$.
If $A \subset X$ is $T$-invariant,
then since $F^{-1}\circ \pi^{-1}(A) = \pi^{-1}\circ T^{-1}(A) = \pi^{-1}(A)$,
we conclude that $\pi^{-1}(A)$ is $F$-invariant.  This implies that
$\tnu(A)$ is 0 or 1.

To prove exponential decay of correlations let $\tf_1, \tf_2 \in C^\alpha(X)$.
Set $f_i = \tf_i \circ \pi$ and
note that
$\int_X \tf_i \, d\tnu = \int_\Delta f_i \circ \pi \, d\nu$ for $i=1,2$.  So
\[
\int_X \tf_1 \, \tf_2 \circ T^n \, d\tnu
= \int_\Delta \tf_1 \circ \pi \, \tf_2 \circ T^n \circ \pi \, d\nu
= \int_\Delta f_1 \, \tf_2 \circ \pi \circ F^n \, d\nu
= \int_\Delta f_1 \, f_2 \circ F^n \, d\nu,
\]
from which exponential decay of correlations follows using
Proposition~\ref{prop:inv measure} and the fact that $f_1, f_2 \in \B_0$.
\end{proof}

%%%%%%%%%%%%%%%%%%%%%%%%%%%%%%%%%%%%%%%%%%%%

\subsection{Equilibrium Principle on the Tower}
\label{tower equilibrium}

First note that since $F$ is mixing,
Property (P3) implies that there exists an $n_0 \in \N$ such that
 $F^n(Z') \supseteq \Delta_0$,
for all $n \geq n_0$ and  $Z' \in \Z^{im}_0$.

Let $\nu_0:= \frac{1}{\nu(\Delta_0)} \nu|_{\Delta_0}$ and define
$S=F^R: \Delta^\infty\cap\Delta_0 \circlearrowleft$.
Let $R_n(x) = \sum_{k=0}^{n-1} R(S^kx)$ be the $n^{\mbox{\scriptsize th}}$
return time
starting at $x$ and let $\M_S$ be the set of $S$-invariant Borel probability
measures on $\Delta^\infty\cap\Delta_0$.

\begin{proposition}
\label{prop:tower equilibrium}
The measure $\nu_0$ is a Gibbs measure for $S$ and $S$ is topologically mixing
on $\Delta^\infty$.  Accordingly,
\[
\sup_{\eta_0 \in \M_S} \left\{ h_{\eta_0}(S)
   + \int_{\Delta_0} \log ((JS)^{-1} \lambda^{-R}) d\eta_0 \right\} = 0 .
\]
and $\nu_0$ is the only nonsingular measure $\eta_0 \in \M_S$ which attains
the supremum.
\end{proposition}

We first prove the following two lemmas.

\begin{lemma}
\label{lem:preimages}
Let $\chi_0$ be the indicator function for $\Delta_0$.
There exists a $k_0 \in \N$ such that for all $k \geq k_0$ and all $x \in \Delta_0$,
\[
\lambda^{-k}\vf^{-1}(x) \Lp^k(\vf \chi_0)(x) \geq \nu(\chi_0)/2.
\]
\end{lemma}

\begin{proof}  Note that $\chi_0 \in \B_0$ so that $\lambda^{-k}\Lp^k(\chi_0) \rightarrow c(\chi_0) \vf$ in
the $\|\cdot\|$-norm.  This means that the functions converge pointwise
uniformly on each level of the tower.  Thus
\[
0 < \nu(\Delta_0) = \lim_{k\rightarrow \infty} \lambda^{-k} \vf^{-1}(x) \Lp^k(\chi_0 \vf)(x)
\]
uniformly for $x \in \Delta_0$.  The uniform convergence implies the existence of the desired $k_0$.
\end{proof}

\medskip
\noindent
The next lemma establishes the Gibbs property for $\nu_0$.
\begin{lemma}
\label{lem:nu bounds}
Let $[i_0, i_1, \ldots, i_{n-1}] \subset \Delta_0$
denote a cylinder set of length $n$ with respect to $S$.  Then there exists
a constant $C>0$ such that for all $n$,
\[
C^{-1} \lambda^{-R_n(y_*)} (JS^n(y_*))^{-1} \leq \nu([i_0, i_1, \ldots, i_{n-1}])
\leq C \lambda^{-R_n(y_*)} (JS^n(y_*))^{-1}
\]
where $y_*$ is an arbitrary point in  $[i_0, i_1, \ldots, i_{n-1}]$ and
$JS^n$ is the Jacobian of $S^n$ with respect to $m$.
\end{lemma}
\begin{proof}
Let $\chi_A$ be the indicator function of  $A:=[i_0, i_1, \ldots, i_{n-1}]$.
Although $\chi_A \notin \B$, we do have $\Lp^k\chi_A \in \B$ for $k \geq n$ since 1-cylinders are in $\B$.
Thus $\nu(\chi_A)$ is characterized by the limit
$\nu(\chi_A) = \lim_k \lambda^{-k} \vf^{-1} \Lp^k(\vf \chi_A)$.
Since this convergence is in the $\| \cdot \|$-norm,
it is uniform for $x \in \Delta_0$.

For $x \in \Delta_0$ and $k \geq R_n(A)$,
\begin{equation}
\label{eq:splitting}
\begin{split}
\Lp^k(\vf \chi_A)(x) & =  \sum_{F^ky = x} \vf(y) \chi_A(y)g_k(y) \\
   & =  \sum_{y \in A, F^ky=x} \vf(y) g_{k-R_n(A)}(F^{R_n}y) g_{R_n}(y) \\
   & =  \sum_{z \in F^{R_n}(A), F^{k-R_n(A)}z = x}
            \vf(y) g_{k-R_n(A)}(z) g_{R_n}(y),
\end{split}
\end{equation}
where in the last line we have used the fact that $F^{R_n(A)}|_A$ is injective.
Note that by \eqref{eq:distortion}, we may replace $g_{R_n}(y)$
by $g_{R_n}(y_*)$ where $y_* \in A$ is an arbitrary point.
Also, since both $y$ and
$F^{R_n}y$ are in $\Delta_0$
and $\delta \leq \vf \leq \delta^{-1}$ on $\Delta_0$, we may estimate \eqref{eq:splitting} by
\begin{equation}
\label{eq:top bound}
\begin{split}
\Lp^k(\vf \chi_A)(x) & \leq C g_{R_n}(y_*) \sum_{z \in F^{R_n}(A), F^{k-R_n(A)}z = x}
                              \vf(z) g_{k-R_n(A)}(z)  \\
             & \leq C g_{R_n}(y_*) \sum_{F^{k-R_n(A)}z = x}\vf(z) g_{k-R_n(A)}(z) \\
             & = C g_{R_n}(y_*) \Lp^{k-R_n(A)}\vf(x)
             \; = \;  C g_{R_n}(y_*) \lambda^{k-R_n(A)}\vf(x).
\end{split}
\end{equation}
Combining this estimate with the definition of $\nu$ and noticing that
$g_{R_n} = (JS^n)^{-1}$,
we have the upper bound,
\[
\nu(A) \leq C (JS^n(y_*))^{-1} \lambda^{-R_n(A)}.
\]

To obtain the lower bound, we again work from equation \eqref{eq:splitting}
and choose $k \geq R_n(A) + n_0 + k_0$.
\begin{equation}
\begin{split}
\Lp^k(\vf \chi_A)(x) & =  \sum_{y \in A, F^ky=x} \vf(y) g_{k-R_n(A)-n_0}(F^{R_n+n_0}y) g_{n_0}(F^{R_n}y) g_{R_n}(y)  \\
   & \geq  \sum_{\substack{z \in F^{R_n+n_0}(A) \cap \Delta_0 \\ F^{k-R_n(A)-n_0}z = x}}
            \vf(y) g_{k-R_n(A)-n_0}(z) g_{n_0}(F^{R_n}y) g_{R_n}(y).
\end{split}
\end{equation}
We again replace $g_{R_n}(y)$ by $g_{R_n}(y_*)$ using \eqref{eq:distortion}. Note also that
$g_{n_0}|_{\Delta_0}$ is bounded below and that $F^{R_n(A)}y \in \Delta_0$.
Since we are only considering $y, z \in \Delta_0$, we know that $\vf(y)$ is proportional to $\vf(z)$.
Thus
\begin{equation}
\label{eq:lower bound}
\begin{split}
\Lp^k(\vf \chi_A)(x) & \geq  C g_{R_n}(y_*) \sum_{\substack{z \in F^{R_n+n_0}(A) \cap \Delta_0 \\ F^{k-R_n(A)-n_0}z = x}}
            \vf(z) g_{k-R_n(A)-n_0}(z) \\
            & = C g_{R_n}(y_*) \sum_{F^{k-R_n(A)-n_0}z=x} \chi_0(z) \vf(z) g_{k-R_n(A)-n_0}(z) \\
            & = C g_{R_n}(y_*) \Lp^{k-R_n(A)-n_0}(\chi_0 \vf)(x).
\end{split}
\end{equation}
where in the second to last line we have used the fact that $F^{R_n+n_0}(A) \supseteq \Delta_0$.
Combining equation~\eqref{eq:lower bound} with Lemma~\ref{lem:preimages}, since $k-R_n(A)-n_0 \geq k_0$, we estimate
\[
 \Lp^k(\vf \chi_A)(x) \geq C g_{R_n}(y_*) \lambda^{k-R_n(A)-n_0} \vf(x) \frac{\nu(\chi_0)}{2}.
\]
The lower bound follows from the definition of $\nu$.
\end{proof}

\begin{proof}[Proof of Proposition \ref{prop:tower equilibrium}]
Lemma~\ref{lem:nu bounds} implies that $\nu_0$ is a Gibbs measure with
potential $\phi = -\log (\lambda^R JS)$.
We define a topology on $\Delta$ using the cylinder sets with respect
to $\Z$ as our basis.
The fact that $S|_{\Delta_0 \cap \Delta^\infty}$ is topologically mixing
follows immediately from the condition that
$F$ be mixing on elements of $\Z_0^{im}$ together with the finite images
condition (P3).  This can be seen as in the proof of
Proposition~\ref{prop:inv measure}(ii) in Section~\ref{invariant measure}.

The formalism of \cite{sarig} completes the proof of the proposition.
Theorem 3 of \cite{sarig} implies that
\begin{equation}
\label{eq:S-equil}
P_G(\phi) = \sup_{\eta_0 \in \M_S} \left\{ h_{\eta_0}(S)
    + \int_{\Delta_0} \phi \, d\eta_0 \right\}
\end{equation}
where $P_G(\phi) = \sup \{ P_{\mbox{\tiny top}} (\phi|_Y) :
Y \subset \Delta_0 \cap \Delta^\infty, \;
\mbox{top.\ mixing finite Markov shift} \}$
is the Gurevich pressure of $\phi$ for $S$.

Lemma~\ref{lem:nu bounds} of this paper combined with
\cite[Theorems 7 and 8]{sarig}
implies that $P_G(\phi) = 0$ and that the supremum is obtained by our
Gibbs measure $\nu_0$.  In addition, $\nu_0$ is the only nonsingular
$S$-invariant probability measure which attains the supremum.
\end{proof}

We now prove an equilibrium principle for $F$ using the one for $S$.

\begin{lemma}
\label{lemma:lyapunov}
Let $\M_F$ be the set of $F$-invariant Borel probability measures on $\Delta$.
For any $\eta \in \M_F$, let
$\eta_0 = \frac{1}{\eta_0(\Delta_0)} \eta|_{\Delta_0}$.  Then
\[
\int_{\Delta_0} \log JS \, d\eta_0 = \int_\Delta \log JF \, d\eta \int_{\Delta_0} R \,d\eta_0 .
\]
\end{lemma}
\begin{proof}
Notice that $\eta_0 \in \M_S$.  For $x \in \Delta_0$,
$JS(x) = JF^R(x) = \Pi_{i=0}^{R(x)-1}JF(F^ix)$. But
$JF(F^ix) =1$ for $i<R(x)-1$, so that $JS(x) = JF(F^{R-1}x)$.
In other words, we have
\begin{equation}
\label{eq:first int}
\int_{\Delta_0} \log JS \, d\eta_0 = \eta(\Delta_0)^{-1} \int_{F^{-1}\Delta_0} \log JF \,d\eta
= \nu(\Delta_0)^{-1} \int_\Delta \log JF \,d\eta.
\end{equation}

Since the measure of a partition element $\zlj \in \Z$ does not change as it
moves up the tower, we have
\[
1 = \sum_{(\ell,j)} \eta(\zlj) = \sum_j \eta(Z_{0,j})R(Z_{0,j}) = \int_{\Delta_0} R \, d\eta.
\]
So by definition of $\eta_0$, we have
\[
 \int_{\Delta_0} R \, d\eta_0 = \eta(\Delta_0)^{-1} \int_{\Delta_0} R \, d\eta = \eta(\Delta_0)^{-1}.
\]
This, together with \eqref{eq:first int}, proves the lemma.
\end{proof}

Since $S = F^R$ is a first return map to $\Delta_0$, the general formula
of Abramov \cite{abramov} implies that $h_\eta(F) = h_{\eta_0}(S) \eta(\Delta_0)$ so that
\begin{equation}
\label{eq:entropy}
h_{\eta_0}(S) = \eta(\Delta_0)^{-1} h_\eta(F) = h_\eta(F) \int_{\Delta_0} R \,d\eta_0.
\end{equation}
Since $ \int_{\Delta_0} R \,d\eta_0 = \eta(\Delta_0)^{-1} \neq 0$
and there is a 1-1 correspondence between measures in $\M_S$
and $\M_F$,
putting equation~\eqref{eq:entropy} and
Lemma~\ref{lemma:lyapunov} together with \eqref{eq:S-equil}, we have
\begin{equation}
\label{eq:F-equilibrium}
\log \lambda = \sup_{\eta \in \M_F} \left\{  h_\eta(F)
  - \int_\Delta  \log JF \, d\eta \right\}.
\end{equation}
Moreover, $\nu$ is the only nonsingular $F$-invariant probability measure which
attains the supremum.  This
completes the proof of Theorem~\ref{thm:F-equilibrium}.

%%%%%%%%%%%%%%%%%%%%%%%%%%%%%%%

\subsection{An Equilibrium Principle for $(T,X)$}
\label{T-equilibrium}

The proof of Theorem~\ref{thm:T-equilibrium} consists simply of
projecting \eqref{eq:F-equilibrium} down to $X$ to get the desired
relation for $T$.

Note that for any $\eta \in \M_F$, we can define
$\teta = \pi_*\eta \in \M_T$.  Then given a
function $\tf$ on $X$, we have
$\int_X \tf \, d\teta = \int_\Delta \tf \circ \pi \, d\eta$.
From the relation $\pi \circ F = T \circ \pi$, we have
\[
J\pi (Fx) JF(x) = JT(\pi x) D\pi(x)
\]
for each $x \in \Delta$.  Thus,
\[
\int_X \log JT \, d\teta = \int_\Delta (\log JF + \log J\pi\circ F - \log J\pi)
 \, d\eta
= \int_\Delta \log JF \, d\eta
\]
since the last two terms cancel by the the $F$-invariance of $\eta$.

The fact that $h_\eta(F) = h_{\teta}(T)$ follows since $\pi$ is at most
countable-to-one (\cite[Proposition 2.8]{buzzi}).  Thus
\[
h_\eta(F) - \int_\Delta \log JF \, d\eta
= h_{\teta}(T) - \int_X \log JT \, d\teta
\]
for each $\eta \in \M_F$, which proves the theorem.

%%%%%%%%%%%%%%%%%%%%%%%%%%%%%%%%%%%%%%%%%%%%%%%%%%%

\medskip
\noindent
Department of Mathematics\\
University of Surrey\\
Guildford, Surrey, GU2 7XH\\
UK\\
\texttt{h.bruin@surrey.ac.uk}\\
\texttt{http://personal.maths.surrey.ac.uk/st/H.Bruin/}

\medskip
\noindent
Department of Mathematics and Computer Science\\
Fairfield University\\
Fairfield, CT 06824 \\
USA\\
\texttt{mdemers@fairfield.edu}\\
\texttt{http://cs.fairfield.edu/$\sim$demers/}

\medskip
\noindent
Department of Mathematics\\
University of Surrey\\
Guildford, Surrey, GU2 7XH\\
UK\\
\texttt{ism@math.uh.edu}\\
\texttt{http://personal.maths.surrey.ac.uk/st/I.Melbourne/research.html}

\end{document}